\documentclass[article]{amsart}

\usepackage[letterpaper, left=2.5cm, right=2.5cm, top = 1in, bottom = 1in]{geometry}

\usepackage{mathtools}

\newtheorem{thm}{Theorem}[section]
\newtheorem{lem}[thm]{Lemma}
\newtheorem*{thm*}{Theorem}

\newtheorem{proposition}[thm]{Proposition}
\newtheorem{corollary}[thm]{Corollary}

\theoremstyle{definition}
\newtheorem{definition}[thm]{Definition}
\newtheorem{example}[thm]{Example}

\newcommand{\Or}{\mathcal{O}}
\newcommand{\M}{\mathcal{M}}

\DeclareMathOperator{\Div}{Div}
\newcommand{\T}{\mathcal{T}}

\newcommand{\Subdiv}{\text{Subdiv}}
\newcommand{\Su}{\mathcal{S}}

\usepackage{subcaption}
\usepackage{multicol}
\usepackage{tikz}
\usepackage{hyperref}
\usepackage{graphicx}
\usepackage{amsthm}
\usepackage{amssymb}
\usepackage{tikz-cd}
\usetikzlibrary{decorations.markings, arrows.meta}
\tikzset{
	mid arrow/.style={postaction={decorate, decoration={
				markings,
				mark=at position .6 with {\arrow{latex}},
	}}},
}
\tikzset{
	bi arrow/.style={postaction={decorate, 
			decoration={
				markings,
				mark=at position .3 with {\arrowreversed{latex}},
				mark=at position .7 with {\arrow{latex}}
	}}},
}

\tikzset{mark size = 10}

\title{Stability conditions and subdivisions of Lawrence polytopes}
\author{Natasha Crepeau}
\address{Department of Mathematics, University of Washington, USA; {\em Present address: Department of Mathematics, Williams College, USA}}
\email{nc20@williams.edu}
\begin{document}
	
	\begin{abstract}
		Let $X$ be a nodal curve, and $G$ be the graph dual to $X$. A stability condition on a graph $G$ is an assignment of integers to the biconnected subsets of vertices of $G$ satisfying some desired properties. Stability conditions yield degeneracy sets, which are certain collections of biconnected subsets of $G$, but degeneracy sets can also be described without a reference stability condition. We show that the degeneracy sets of a graph $G$ correspond to the single-element extensions of the graphic matroid $M(G)$. Some single-element extensions of the graphic matroid $M(G)$ can be oriented to be single-element extensions of the oriented graphic matroid $\M(G)$. Single-element extensions of oriented matroids are in bijection with subdivisions of the Lawrence polytope of the dual oriented matroid. We construct stability conditions on $G$ from subdivisions of the Lawrence polytope of the cographic matroid $\M^\ast(G)$. Finally, we show that any degeneracy set $\mathcal{D}$ of a graph $G$ that has the set of all biconnected subsets of $G$ as a lower bound in the poset of degeneracy sets corresponds to an orientable extension of $M(G)$. 
	\end{abstract}
	\maketitle
	\section{Introduction}
	\renewcommand*{\thethm}{\Alph{thm}}
	
	A \textit{stability condition} $\mathfrak{n}$ on a nodal curve $X$ can be viewed as an assignment of integers $\mathfrak{n}_W$ to the biconnected subsets $W$ of the vertices of graph $G$ dual to $X$ satisfying certain properties, as in, for example, \cite{pagani2022geometry}, \cite{pagani2023stability}, \cite{viviani2023new} and \cite{fava-pagani-viviani2024}. These stability conditions are used to construct compactified Jacobians on nodal curves. For biconnected subset $W$ and its complement $W^c$, $\mathfrak{n}$ satisfies $\mathfrak{n}_W + \mathfrak{n}_{W^c} + |E(W,W^c)| - d \in \{0,1\}$, where $|E(W,W^c)|$ is the number of edges with one endpoint in $W$ and the other in $W^c$, and $d$ is the degree of the stability condition. The set of all biconnected subsets $W$ such that $\mathfrak{n}_W + \mathfrak{n}_{W^c} + |E(W,W^c)| - d = 0$ form the \textit{degeneracy set} $\mathcal{D}(\mathfrak{n})$ of the stability condition. Degeneracy sets of the graph $G$ can also be defined without a reference stability condition.
	
	In this paper, we study stability conditions on a graph $G$ and degeneracy sets of the graph. To do so, we consider the unoriented graphic matroid $M(G)$ of $G$ and the oriented graphic matroid $\M(G)$ of $G$. A \textit{single-element extension} of a matroid $M$ is a matroid $M'$ such that $M = M' \backslash e$ for some element $e$ in $M'$. Single-element extensions of unoriented matroids were fully characterized by Crapo \cite{crapo-extensions}, while single-element extensions of oriented matroids were fully characterized by Las Vergnas \cite{lasvergnas}. In both characterizations, single-element extensions are described by which cocircuits of the base matroid are extended to include $e$. When $\M$ is an oriented regular matroid, the single-element extensions of $\M$ are in bijection with \textit{subdivisions} of the \textit{Lawrence polytope} corresponding to the dual matroid $\M^\ast$ \cite[Theorem 4.14]{triangulations-oriented}. Although graphic matroids are always \textit{orientable}, not all extensions of $M(G)$ correspond to some extension of $\M(G)$. Our first result shows that the set of degeneracy sets of the graph $G$ are in bijection with the single-element extensions of the graphic matroid $M(G)$.
	
	\begin{lem}[Lemma \ref{lem:degsets-subclasses}]
		$\mathcal{D}$ is a degeneracy set of a graph $G$ if and only if there exists a single-element extension of $M$ by an element $p$ such that $\{C^\ast \in C^\ast(M):\sigma^\ast(C^\ast)=0\} = C^\ast(\mathcal{D})$. 
	\end{lem}
	We refer to the degeneracy sets corresponding to orientable extensions of $M(G)$ as \textit{orientable} degeneracy sets. The \textit{extension poset} of a matroid orders extensions by inclusion; however in \cite{fava-pagani-viviani2025}, the authors define the poset of degeneracy sets, $\text{Deg}(G)$, with a relation stronger than inclusion \cite[Definition 4.18]{fava-pagani-viviani2025}.
	
	In previous work, we constructed nondegenerate stability conditions from triangulations of the Lawrence polytope $\Lambda(\M^\ast(G))$ \cite[Theorem 5.1]{crepeau}. Our second main result in this paper constructs a degenerate stability condition from a subdivision of the Lawrence polytope $\Lambda(\M^\ast(G))$. Every subdivision of a polytope can be refined to a triangulation of that polytope, and we use that triangulation and its corresponding stability condition to construct a stability condition corresponding to the subdivision.
	\begin{thm}[Theorem \ref{thm:stability-condition-through-refinement}]
		Let $\Su$ be a subdivision of the Lawrence polytope $\Lambda(\M^\ast(G))$. Let $\mathcal{T}$ be a triangulation of $\Lambda(\M^\ast(G))$ that refines $\Su$, and consider the associated stability condition $\mathfrak{n}^\T$. Let $\sigma_\T^\ast$ be the localization of the extension of $\M(G)$ corresponding to $\T$, and let $\sigma^\ast_\Su$ be the localization of the extension of $\M(G)$ corresponding to $\Su$. Let $\vec{C}^\ast_W$ be the cocircuit $E(W,W^c)$ of $\M(G)$ oriented towards $W$. 
		
		Define $\mathfrak{n}^\Su: \text{BCon}(G) \to \mathbb{Z}$ by
		\[\mathfrak{n}^\Su_W = \begin{cases} \mathfrak{n}^\T_W - 1 & W \in \mathcal{D}_\Su \text{ and } \sigma_\T^\ast(\vec{C}^\ast_W) = +, \\
			\mathfrak{n}^\T_W & \text{otherwise.}
		\end{cases}\]
		
		Then, $\mathfrak{n}^\Su$ is a degenerate stability condition of degree $g(G) - 1$ such that $\mathcal{D}(\mathfrak{n}^\Su) = \{W:\sigma_\Su^\ast(\pm\vec{C}^\ast_W)=0\}$.
	\end{thm}
	Additionally, every regular subdivision of a polytope can be refined to a regular triangulation of that polytope, so regular subdivisions of $\Lambda(\M^\ast(G))$ yield \textit{classical} stability conditions, which are a particular kind of stability condition (Corollary \ref{cor:reg->clasical}). As a consequence, all orientable degeneracy sets are the degeneracy set of some stability condition (Corollary \ref{cor:stab->odeg is surj}). 
	
	Finally, there exists stability conditions that have degeneracy sets that are non-orientable (Example \ref{ex:not-all-extensions-orientable}), but it is unknown if all degeneracy sets are the degeneracy set of some stability condition \cite[Question 4.23]{fava-pagani-viviani2025}. In our third main result, we show that the orientable degeneracy sets of $G$ form an upper subset of $\text{Deg}(G)$, and that this subset has one minimal element, the set of all biconnected subsets of $V(G)$.
	
	\begin{thm}[Theorem \ref{thm:orientable-iff-bcon}]
		The degeneracy set $\mathcal{D}$ is orientable if and only if $\mathcal{D}\geq \text{BCon}(G)$ in $\text{Deg}(G)$. 
	\end{thm}
	
	One remaining question is to characterize the minimal non-orientable elements in $\text{Deg}(G)$ for a graph $G$, and construct stability conditions whose degeneracy sets are the non-orientable elements in $\text{Deg}(G)$.
	
	\subsection*{Acknowledgements} This article is based on Chapter 4 of the author's PhD thesis \cite{crepeau-thesis}. The author was partially supported by NSF CAREER DMS-2044564. The author thanks Farbod Shokrieh for his guidance and advice during this project.  Additionally, the author would like to thank Gaku Liu, Lilla T\'othm\'er\'esz, and Cameron Wright for the helpful discussions. 
	\renewcommand*{\thethm}{\arabic{section}.\arabic{thm}}
	\section{Preliminaries}
	
	\subsection{Graphs}
	Throughout, we let $G$ be a loopless connected graph with vertices $V(G)$ and edges $E(G)$. The \textit{genus} of $G$ is $g(G) = |E(G)| - |V(G)| + 1$. A \textit{cycle} is a walk along the graph $G$ that starts and ends at the same vertex without using any edge or any vertex twice, and a \textit{cocycle}, or \textit{bond}, is a minimal subset of edges $E'$ such that $G\backslash E'$ has two connected components.
	
	A \textit{subgraph} of $G$ is a graph $G'$ with vertices $V(G') \subseteq V(G)$ and edges $E(G') \subseteq E(G)$. Subgraphs of connected graphs are not necessarily connected. A subgraph $G'$ is \textit{spanning} if $V(G') = V(G)$.  Given a subset $W$ of the vertices $V(G)$, the \textit{induced subgraph} $G[W]$ is the subgraph of $G$ containing all edges between any vertices in $W$. A subset $W$ is \textit{nontrivial} if $W \neq \emptyset$ and $W \neq V(G)$. A subset $W$ is \textit{biconnected} if both $G[W]$ and $G[W^c]$ are connected; if only $G[W]$ is connected, we say that $W$ is \textit{connected}. For disjoint subsets $W_1$ and $W_2$ of $V(G)$, $E(W_1, W_2)$ is the set of edges connecting a vertex in $W_1$ to a vertex in $W_2$; equivalently, it is the set of edges between $G[W_1]$ and $G[W_2]$. Every cocycle of $G$ can be written as $E(W,W^c)$, where $W$ is a biconnected subset. Given an edge $e=vw$, the \textit{deletion} of $e$ from $G$ is the graph $G\backslash \{e\} = (V, E\backslash \{e\})$. The \textit{contraction} of $e$ is the graph $G/e = (V', E')$, where $V' = V \backslash \{v,w\} \cup \{v'\}$ for a new vertex $v'$, and $E'$ is all edges in $E$ that were not incident to $v$ or $w$ in $G$, and all edges that were incident to $v$ or $w$ that are now incident to $v'$ instead. 
	
	A \textit{tree} is a graph without any cycles. A \textit{spanning tree} $T$ of $G$ is a connected spanning subgraph of $G$ that is a tree. A \textit{spanning forest} is a disconnected spanning subgraph of $G$ made up of trees. The collection of all spanning trees of a graph $G$ is denoted by $\mathcal{ST}(G)$. Given a spanning tree $T$ and an edge $e \in E(G \backslash T)$, the \textit{fundamental cycle} $C(T,e)$ is the unique cycle in $T \cup \{e\}$. If $e \in E(T)$, $T \backslash e$ has two connected components with vertices $W$ and $W^c$. The \textit{fundamental cocycle} $C^\ast(T,e)$ is the unique cocycle $E(W,W^c)$ partitioning $V(G)$ into $W$ and $W^c$. 
	
	A \textit{divisor} $D$ on a graph $G$ is an assignment of integers $D(v)$ to the vertices of $G$. The \textit{degree} of a divisor $D$ is $\deg(D) = \sum_{v \in V(G)} D(v)$. We let $\Div^d(G)$ be the collection of all divisors of degree $d$ on $G$, and $\text{Div}(G)$ denote the collection of all divisors on $G$. An \textit{orientation} $\Or$ of a graph is a choice of direction for every edge of $G$. Given an orientation $\Or$ of a graph $G$, we can associate a divisor $D_\Or$ to it by setting $D_\Or(v) = \text{indegree}_\Or(v) - 1$ for each $v \in V(G)$, where $\text{indegree}_\Or(v)$ is the number of edges incident to $v$ pointing towards $v$. The degree of $D_\Or$ is always $g(G)-1$. If $D = D_\Or$ for some orientation $\Or$, we say that $D$ is \textit{orientable}. A \textit{fourientation} is a generalization of orientations that allows unoriented edges and edges oriented in both directions (bioriented), in addition to one-way oriented edges. For a fourientation $\vec{F}$, $-\vec{F}$ is the fourientation with all one-way oriented arcs reversed, while $\vec{F}^c$ reverses one-way oriented arcs and switches unoriented and bioriented arcs.
	
	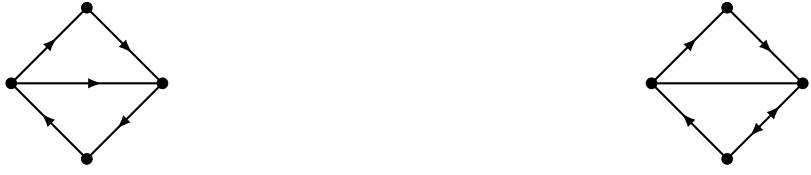
\begin{figure}
		\centering
		\begin{multicols}{2}
			\begin{tikzpicture}
				\draw[thick, mid arrow] (0,1) to (1,0);
				\draw[thick, mid arrow] (1,0) to (0,-1);
				\draw[thick, mid arrow] (0,-1) to (-1,0);
				\draw[thick, mid arrow] (-1,0) to (0,1);
				\draw[thick, mid arrow] (-1, 0) to (1,0);
				
				\filldraw[black] (0,1) circle (2pt);
				\filldraw[black] (1,0) circle (2pt);
				\filldraw[black] (0,-1) circle (2pt);
				\filldraw[black] (-1,0) circle (2pt);
			\end{tikzpicture}
			
			\begin{tikzpicture}
				\draw[thick, mid arrow] (0,1) to (1,0);
				\draw[thick, bi arrow] (1,0) to (0,-1);
				\draw[thick, mid arrow] (0,-1) to (-1,0);
				\draw[thick, mid arrow] (-1,0) to (0,1);
				\draw[thick] (-1, 0) to (1,0);
				\filldraw[black] (0,1) circle (2pt);
				\filldraw[black] (1,0) circle (2pt);
				\filldraw[black] (0,-1) circle (2pt);
				\filldraw[black] (-1,0) circle (2pt);
			\end{tikzpicture}
			
		\end{multicols}
		\caption{An orientation and a fourientation of a graph G.}
		\label{fig: orientation ex}
	\end{figure}
	
	\subsection{Matroids}
	Let $E$ be a finite set, which we call the \textit{ground set}. Let $\mathcal{I}$ be a family of subsets of $E$, called \textit{independent sets}, that satisfy the following axioms:
	\begin{enumerate}
		\item $\emptyset \in \mathcal{I}$
		\item If $I \in \mathcal{I}$ and $J \subset I  \subset E$, then $J \in \mathcal{I}$.
		\item If $I, J \in \mathcal{I}$ and $|I| > |J|$, then there exists $x \in I \backslash J$ such that $J \cup \{x\} \in \mathcal{I}$. 
	\end{enumerate}
	Then, $M = (E, \mathcal{I})$ is a \textit{matroid} (see, e.g. \cite{oxley}). The largest (with respect to inclusion) independent sets of a matroid are called \textit{bases}. By the third axiom, any two bases have the same size, which we call the \textit{rank} of $M$. A \textit{loop} of a matroid is an element $e \in E$ such that $e \notin \mathcal{I}$, and a \textit{coloop} is an element $e \in E$ such that $e \in B$ for every base $B$ of $M$. For $e \in E$, the \textit{deletion} of $e$ from $M$ is the matroid $M\backslash e = (E \backslash \{e\}, \mathcal{I}')$, where $\mathcal{I}' = \{I:I \in \mathcal{I} \text{ and } e\notin I\}$, and the \textit{contraction} of $e$ is the matroid $M/e = (E\backslash\{e\}, \mathcal{I}'')$, where $\mathcal{I}'' = \{I\backslash \{e\} : I \in \mathcal{I} \text{ and } e \in I\}$. A \textit{minor} of $M$ is any matroid obtained from $M$ through a sequence of deletions and contractions. A \textit{circuit} $C$ of $M$ is a minimally dependent subset of the ground set $E$. A matroid $M$ has a dual $M^\ast$, whose basis elements $B^\ast$ are the complements of the basis elements $B$ of $M$ in $E$. The \textit{cocircuits} $C^\ast$ of $M$ are the circuits of the dual matroid $M^\ast$.    For a basis $B$ and $e \notin B$, the \textit{fundamental circuit} $C(B,e)$ is the unique circuit contained in $B \cup \{e\}$, and for $e \in B$ the \textit{fundamental cocircuit} $C^\ast(B,e)$ is the unique cocircuit in $(E\backslash B) \cup \{e\}$.
	
	For a connected graph $G$, we consider the (unoriented) graphic matroid $M(G)$. The ground set $E$ of $M(G)$ is the set of edges of the graph $E(G)$, and the independent sets are the subgraphs of $G$ that do not contain a cycle. The bases of $M(G)$ are the spanning trees of $G$, and the rank of $M(G)$ is $r(M(G)) = |V(G)| - 1$. The circuits of $M(G)$ are the subsets of edges corresponding to cycles of $G$, and the cocircuits of $M(G)$ are the subsets of edges corresponding to cocycles of $G$. If $G$ is planar, then the dual $M^\ast(G)$ is the graphic matroid of the planar dual of $G$; however is $G$ is not planar, there is no graph interpretation of $M^\ast(G)$. The rank of $M^\ast(G)$ is the genus $g(G)$ of the graph $G$.
	
	Graphic matroids are an example of \textit{regular matroids}, which are matroids that can be defined from a totally unimodular matrix $U$. The dimensions of $U
	$ are $r \times n$, where $r$ is the rank of the matroid and $n = |E|$. For a graphic matroid $M(G)$, we take the matrix $U$ to be the oriented incidence matrix of $G$ with a row removed. For any loopless regular matroid $M$ that is represented by a totally unimodular matrix $U$, $U$ can be written as $U = [I_r | L]$. Then, the dual matroid $M^\ast$ can be represented by the matrix $U^\ast = [-L^T | I_{n-r}]$, as shown in Theorem 2.2.8 of \cite{oxley}.
	
	In this paper, we also work with oriented matroids. Oriented matroids are defined on a set of edges $E$. An \textit{arc} is an oriented edge $\vec{e}$. A \textit{signed subset} of $E$ is a set $X$ such that $X \in \{-,0,+\}^E$, where $X_e$ is the sign of the edge $e$. The \textit{support} of a signed set $X$ is $\underline{X} = \{ e | X_e \neq 0\}$. 
	
	One way to construct an oriented matroid on a set $E$ is through defining its \textit{signed circuits}.
	\begin{definition}\cite[Definition 3.2.1]{bjorner}\label{def:circuit-axioms}
		A collection of $\mathcal{C}$ of signed subsets of a set $E$ is the set of \textit{signed circuits} of an \textit{oriented matroid} on $E$ if and only if it satisfies the following axioms:
		\begin{enumerate}
			\item $\emptyset \notin \mathcal{C}$,
			\item if $X \in \mathcal{C}$, then $-X \in \mathcal{C}$,
			\item for all $X, Y \in \mathcal{C}$, if $\underline{X} \subset \underline{Y}$, then $X=Y$ or $X = -Y$,
			\item for all $X, Y \in \mathcal{C}$, $X \neq -Y$, and $e \in X^+ \cap Y^-$ there is a $Z \in \mathcal{C}$ such that $Z^+ \subseteq (X^+ \cup Y^+)\backslash\{e\}$ and $Z^- \subseteq (X^-\cup Y^-)\backslash\{e\}$.
		\end{enumerate}
	\end{definition}
	
	For an oriented matroid $\M$, we denote the underlying unoriented matroid by $\underline{\M}$. An unoriented matroid $M$ is \textit{orientable} is there exists an oriented matroid $\mathcal{M}$ such that $\underline{\mathcal{M}} = M$. For an orientable matroid $M$, the support of the signed circuits of $\M$ are exactly the circuits of $M$. 
	
	We can also define minors of oriented matroids using the set of signed circuits. Consider $e \in E$. Then, the deletion of $e$ from the oriented matroid $\M$, denoted $\M\backslash e$, is defined by the signed circuits $\mathcal{C}' = \{\vec{C} \in \mathcal{C} | \underline{\vec{C}} \subseteq E \backslash \{e\}\}$, and the contraction of $e$ in the oriented matroid $\M$, denoted $\M/e$, is defined by its signed circuits, which are the inclusion-wise minimal elements in $\{\vec{C} \backslash \{e\}: \vec{C} \in \mathcal{C}\}$.
	
	For an undirected graph $G$, we can orient the matroids $M(G)$ and $M^\ast(G)$ by endowing $G$ with a choice of orientation $\Or$. Then, for each cycle $C$ of $G$, let $\vec{C}$ denote a choice of orientation of $C$, where every arc in $\vec{C}$ either agrees with the arc in $\Or$ or is oriented in the opposite direction in $\Or$. We can then view $\vec{C}$ as a signed subset of $E$, where 
	\[ \vec{C}_e = \begin{cases}
		+ & \text{if } \vec{e} \in \vec{C} \text{ and } \vec{e}\in \Or \\
		- & \text{if } \vec{e} \in \vec{C} \text{ and } -\vec{e}\in \Or \\
		0 & \text{otherwise}.
	\end{cases}\]
	Every signed cycle $\vec{C}$ can then be described by its positive part $\vec{C}_+ : \{e : \vec{C_e} = +\}$ and its negative part $\vec{C}_{-} = \{e : \vec{C_e} = -\}$. We can write any signed cycle $\vec{C}$ as $\vec{C} = (\vec{C}_+, \vec{C}_{-})$. The collection of signed cycles of $G$ satisfy the circuit axioms in Definition \ref{def:circuit-axioms}, so the collection of signed cycles are the set of (signed) circuits of the oriented matroid $\M(G)$. The collection of all signed cocycles also yields the set of signed cocircuits of the oriented matroid $\M(G)$, and we denote signed cocircuits by $\vec{C}^\ast$. In \cite{ding2023framework}, the definition of fourientation is also extended to regular matroids, where a fourientation is a subset of the set of all arcs of $\M$; as in the graph case, $e \in E$ can be bioriented, unoriented, or one-way oriented.  
	
	Oriented regular matroids have a \textit{circuit-cocircuit reversal system}, introduced in \cite{gioan-circuit-cocircuit}, where an orientation $\Or$ of $\M$ is equivalent to another orientation $\Or'$ if $\Or'$ is reachable from $\Or$ through a sequence of circuit and cocircuit reversals. The number of circuit-cocircuit reversal classes is the number of bases of the matroid. 
	\subsection{Single-element extensions of matroids} \label{subsec:extension}
	Let $M$ be an (unoriented or oriented) matroid on a ground set $E$. Then, a \textit{single-element extension} of $M$ by some element $p$ is a matroid $M'$ on the ground set $E \cup \{p\}$ such that $M'\backslash p = M$. Dually, a \textit{single-element lifting} of $M$ is a matroid $\tilde{M}$ on the ground set $E \cup \{p\}$ such that $\tilde{M} / p = M$. We note that single-element extensions of $M$ are dual to \textit{single-element liftings} of $M^\ast$. A single-element extension is \textit{nontrivial} if $r(M') = r(M)$. We assume all of the single-element extensions we consider are nontrivial.
	
	In both the unoriented and oriented case, extensions are characterized by how they extend the cocircuits of starting matroid. First, we present the characterization of single-element extensions of unoriented matroids from \cite{crapo-extensions}. 
	\begin{proposition}[{\cite{crapo-extensions},\cite[Proposition 7.1.2]{bjorner}}]
		Let $M$ be a matroid on the ground set $E$, and denote the set of cocircuits of $M$ by $C^\ast(M) \subseteq \{0,1\}^E$, given by their incidence vectors. Let $M' = M \cup\{p\}$ be an extension of $M$ by a new element $p$. Let $C^\ast(M')$ be the cocircuits of $M'$. Then for every cocircuit $C^\ast \in C^\ast(M)$, there is a unique cocircuit of $M'$ contained in $C^\ast \cup \{p\}$, that is, there is a unique function 
		\[\sigma^\ast:C^\ast(M) \to \{0,1\}\]
		such that $\{(C^\ast, \sigma^\ast(C^\ast)): C^\ast \in C^\ast(M)\} \subseteq C^\ast(M')$. 
		
		Furthermore, all the cocircuits of $M'$ can be explicitly described using $\sigma^\ast$. Thus, $M'$ is uniquely determined by $\sigma^\ast$. 
	\end{proposition}
	
	However, not every function $\sigma^\ast$ corresponds to a single-element extension of a matroid $M$. Two subsets $A$ and $B$ of the ground set $E$ form a \textit{modular pair} of $M$ if $r(A) + r(B) = r(A \cup B) + r(A \cap B)$. Two cocircuits $C_1^\ast$ and $C_2^\ast$ of $M$ are a \textit{modular pair of cocircuits} if they are a modular pair of the dual matroid $M^\ast$.
	\begin{thm}[{\cite{crapo-extensions},\cite[Theorem 7.1.3]{bjorner}}] \label{thm:sigs-unoriented extensions}
		Let $M$ be a matroid. Then, for a function $\sigma^\ast:C^\ast(M) \to \{0,1\}$, which assigns a signature $\sigma^\ast(C^\ast) \in \{0,1\}$ to every cocircuit $C^\ast$ of $M$, the following are equivalent:
		\begin{enumerate}
			\item There exists a single-element extension $M'$ of $M$ by an element $p$ such that $\{(C^\ast, \sigma^\ast(C^\ast)): C^\ast \in C^\ast(M)\} \subseteq C^\ast(M')$.
			\item $\sigma^\ast$ defines a single-element extension on every rank $2$ contraction of $M$.
			\item If a modular pair of cocircuits has signature $0$, then so has every other cocircuit in their union. 
		\end{enumerate}
	\end{thm}
	
	Therefore, we can describe single-element extensions by the cocircuits $C^\ast$ with signature $\sigma^\ast(C^\ast) = 0$. When $\sigma^\ast$ is the signature of a single-element extension of $M$ as in Theorem \ref{thm:sigs-unoriented extensions}, the collection $\{C^\ast \in C^\ast(M):\sigma^\ast(C^\ast) = 0\}$ is a \textit{linear subclass of cocircuits}. The set of linear subclasses of cocircuits of $M$, when partially ordered by inclusion, form the \textit{extension lattice} of $M$ \cite{cheung}. 
	For graphic matroids, the single-element extensions of $M(G)$ can also be characterized by cobiased graphs \cite{slilaty}. 
	
	Now, consider an oriented matroid $\M$. Single-element extensions of $\M$ are characterized by how they extend the signed cocircuits of $\M$.
	\begin{proposition}[\cite{lasvergnas}, {\cite[Proposition 7.1.4]{bjorner}}]
		Let $\M$ be an oriented matroid on the ground set $E$, and denote the signed cocircuits of $\M$ by $\mathcal{C}^\ast(\M)$. Let $\M'$ be a single-element extension of $\M$. Then for every signed cocircuit $\vec{C}^\ast \in \mathcal{C}^\ast(\M)$, there is a unique way to extend $\vec{C}^\ast$ to a cocircuit of $\M'$: there is a unique function
		\[\sigma^\ast:\mathcal{C}^\ast(\M) \to \{+, -, 0\}\]
		such that $(\vec{C}^\ast, \sigma^\ast(\vec{C}^\ast))$ is a signed cocircuit of $\M'$. Furthermore, this $\sigma^\ast$ satisfies $\sigma^\ast(-\vec{C}^\ast) = - \sigma^\ast(\vec{C}^\ast)$. Finally, $\M'$ is uniquely determined by $\sigma^\ast$. 
	\end{proposition}
	
	Every single-element extension of $\M$ defines a single-element extension of the underlying unoriented matroid $\underline{\M}$. If $\sigma^\ast$ determines an extension of an oriented matroid $\M$, we say that $\sigma^\ast$ is a \textit{localization}. Oriented matroids of rank $2$ can always be represented by a central hyperplane arrangement, where each line corresponds to a cocircuit of the matroid, so a localization $\sigma^\ast$ can be described by its behavior on all the rank $2$ minors of $\M$ (see \cite[Example 7.1.7]{bjorner}).
	
	\begin{thm}[\cite{lasvergnas}, {\cite[Theorem 7.1.8]{bjorner}}] \label{thm::or-ext-equivalences}
		Let $\M$ be an oriented matroid, and $\sigma^\ast : \mathcal{C}^\ast(\M) \to \{+, -, 0\}$ a signature satisfying $\sigma^\ast(-\vec{C}^\ast) = -\sigma^\ast(\vec{C}^\ast)$ for all $\vec{C}^\ast \in \mathcal{C}^\ast(\M)$ Then, the following statements are equivalent:
		
		\begin{enumerate}
			\item $\sigma^\ast$ is a localization: there exists a single-element extension $\M'$ of $\M$ such that
			$$\{(\vec{C}^\ast, \sigma^\ast(\vec{C}^\ast)) : \vec{C}^\ast \in \mathcal{C}^\ast(\M)\} \subseteq \mathcal{C}^\ast(\M').$$
			\item $\sigma^\ast$ defines a single-element extension on every contraction of $\M$ of rank $2$. 
			\item The signature $\sigma^\ast$ produces none of the three excluded subconfigurations (minors) of rank $2$ on three elements, as given in Figure \ref{fig:forbiddenconfigs}. Each ray corresponds to a cocircuit of the rank $2$ minor. The signs on each ray from the origin represent the sign of the new element $p$ in $\M' = \M \cup \{p\}$. 
			
		\end{enumerate}
		
	\end{thm}
	
	\begin{figure}
		\centering
		\begin{subfigure}{.3\textwidth}
			\centering
			\begin{tikzpicture}
				\draw (0,1.25)--(0,-1.25);
				\draw (-1,1)--(1,-1);
				\draw (1,1)--(-1,-1);
				\node[anchor = south] at (0,1.25) {$+$};
				\node[anchor = south west] at (1,1) {$-$};
				\node[anchor = south east] at (-1,1) {$-$};
				\node[anchor = north] at (0,-1.25) {$-$};
				\node[anchor = north west] at (1,-1) {$+$};
				\node[anchor = north east] at (-1,-1) {$+$}; 
			\end{tikzpicture}
		\end{subfigure}
		\begin{subfigure}{.3\textwidth}
			\centering
			\begin{tikzpicture}
				\draw (0,1.25)--(0,-1.25);
				\draw (-1,1)--(1,-1);
				\draw (1,1)--(-1,-1);
				
				\node[anchor = south] at (0,1.25) {$-$};
				\node[anchor = south west] at (1,1) {$+$};
				\node[anchor = south east] at (-1,1) {$0$};
				\node[anchor = north] at (0,-1.25) {$+$};
				\node[anchor = north west] at (1,-1) {$0$};
				\node[anchor = north east] at (-1,-1) {$-$};
			\end{tikzpicture}
		\end{subfigure}
		\begin{subfigure}{.3\textwidth}
			\centering
			\begin{tikzpicture}
				\draw (0,1.25)--(0,-1.25);
				\draw (-1,1)--(1,-1);
				\draw (1,1)--(-1,-1);
				\node[anchor = south] at (0,1.25) {$-$};
				\node[anchor = south west] at (1,1) {$0$};
				\node[anchor = south east] at (-1,1) {$0$};
				\node[anchor = north] at (0,-1.25) {$+$};
				\node[anchor = north west] at (1,-1) {$0$};
				\node[anchor = north east] at (-1,-1) {$0$};
			\end{tikzpicture}
		\end{subfigure}
		\caption{Forbidden subconfigurations for signatures of cocircuits}
		\label{fig:forbiddenconfigs}
	\end{figure}
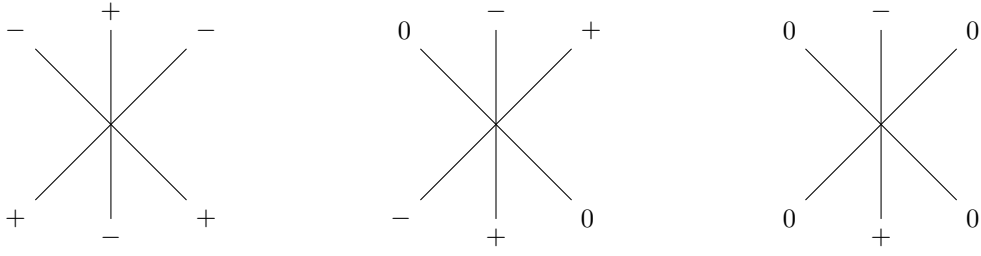
	
	Finally, we note that not every unoriented extension of an orientable matroid corresponds to an oriented extension of that oriented matroid. 
	
	\begin{example}\label{ex:not-all-extensions-orientable}
		Consider $K_4$, the complete graph on $4$ vertices. We label those vertices by $V(K_4) = \{v_1, v_2, v_3, v_4\}$. All graphic matroids are orientable. Consider the linear subclass of cocircuits of $M(K_4)$ given by $$\mathcal{L} = \{E(v_1v_2, v_3v_4), E(v_1v_3, v_2v_4), E(v_1v_4, v_2v_3)\},$$ which corresponds to a single-element extension of $M(K_4)$. However, this extension of $M(K_4)$ yields a matroid that is not orientable \cite[Example 3.3]{oriented-graphic-extensions}, so the extension cannot be oriented.
	\end{example}
	
	\subsection{Polytopes and subdivisions}
	A \textit{polytope} $\mathcal{P}$ is the convex hull of a finite set of points $V$ in $\mathbb{R}^n$. A \textit{face} of a polytope is its intersection with any hyperplane that does not cross the relative interior of $P$. Faces of dimension $0$ are called \textit{vertices}. A $d$-dimensional \textit{simplex} is a polytope of dimension $d$ with exactly $d+1$ vertices. 
	
	A \textit{subdivision} $\Su$ of a polytope $\mathcal{P} \subset \mathbb{R}^d$ is a collection of polytopes, called \textit{cells}, whose union is $\mathcal{P}$, and any two distinct cells intersect at a (possibly empty) common face. A \textit{triangulation} is a subdivision where every cell is a $d$-dimensional simplex. 
	We denote the set of all subdivisions of a polytope $\mathcal{P}$ as $\Subdiv(\mathcal{P})$. We recall the following definitions from \cite{triangulations-book}. A subdivision $\mathcal{S}$ is a \textit{refinement} of the subdivision $\Su'$ if for every cell $c \in \Su$, there exists a cell $c' \in \Su'$ such that $c \subseteq c'$. Refinement induces a partial order on $\Subdiv(\mathcal{P})$. The minimal elements of the poset $(\Subdiv(\mathcal{P}), \preccurlyeq)$ are the triangulations of $\mathcal{P}$, and the unique maximal element is given by the trivial subdivision of $\mathcal{P}$. 
	
	One special kind of subdivision is a \textit{regular subdivision}. Let $V$ be the vertices of a polytope $\mathcal{P} \subset \mathbb{R}^d$, and let $\omega: V \to \mathbb{R}$ be a function that lifts $V$ to $\mathbb{R}^{d+1}$ by sending $V$ to
	\[V_\omega \coloneqq \{(v, \omega(v)) : v \in V \}.\]
	The lower faces of the convex hull of $V_\omega$ form a subdivision of $\mathcal{P}$, which we call a \textit{regular subdivision}. Every subdivision of a polytope can be refined to a triangulation of that polytope; furthermore, every regular subdivision of a polytope can be refined to a regular triangulation of that polytope \cite[Corollary 2.3.18]{triangulations-book}.
	
	\subsection{Lawrence polytopes and connections to single-element extensions}
	Every oriented matroid $\M$ of rank $r$ on $n$ elements can be lifted to an associated oriented matroid $\Lambda(\M)$ of rank $n+r$ on $2n$ elements through a construction of Jim Lawrence (see \cite{lawrencepolytopes}). If $\M$ is a loopless matroid and can be realized by a matrix $M_{r\times n}$, the \textit{Lawrence matrix} of $\M$ is given by
	$$\Lambda(\M) = \begin{pmatrix}
		M_{r\times n} & \mathbf{0} \\
		I_{n\times n} & I_{n\times n}
	\end{pmatrix}.$$
	We label the columns of the Lawrence matrix as $P_1,...,P_n,P_{-1},...,P_{-n}$. The \textit{Lawrence polytope} of $\M$, which we also denote by $\Lambda(\M)$, is the convex hull of the points $P_1,...,P_n,P_{-1},...,P_{-n}$. Each column $P_i$ of $M_{r\times n}$ corresponds to an arc $\vec{e_i}$ of $\M$, and the columns $P_{-i}$ correspond to the arc $\vec{e_{-i}} = -\vec{e_i}$. The vertices of $\Lambda(\M)$, which is the ground set of $\Lambda(\M)$ when viewed as an oriented matroid, are the set $E \cup \overline{E}$, where $E$ is the ground set $E$ of $\M$ and $\overline{E} = \{-\vec{e}:\vec{e} \in E\}$. 
	
	As an oriented matroid, $\Lambda(\M)$ inherits its circuits, cocircuits, and bases from $\M$. If $B$ is an unsigned subset of $E$ and $A \subseteq B$, the \textit{reorientation} of $B$ at $A$ is $_AB \coloneqq (B\backslash A)\cup\overline{A}$. If $X = (X_+, X_-)$, and $Y \subseteq X_+ \cup X_-$ is a signed subset of $E$, the \textit{reorientation} of $X$ at $Y$ is $_YX \coloneqq ((X_+\backslash Y\cup \overline{(X_- \cap A)}),( X_-\backslash A, \overline{(X_+ \cap A)}))$. 
	
	\begin{lem}[{\cite[Lemma 4.11]{triangulations-oriented}}]
		Let $\Lambda(\M)$ be the Lawrence polytope associated with an oriented matroid $\M$. 
		\begin{enumerate}
			\item The set of circuits of $\Lambda(\M)$ is
			\[\mathcal{C}(\Lambda(\M)) \coloneqq \{(C_+\cup \overline{C_-}, C_- \cup \overline{C_+}):(C_+, C_-) \in \mathcal{C}(\M)\}.\]
			\item The set of cocircuits of $\Lambda(\M)$ is 
			\[\mathcal{C}^\ast(\Lambda(\M)) \coloneqq \{_A\vec{C}^\ast: \vec{C^{\ast}} \in \mathcal{C}^\ast(\M), A\subseteq \vec{C}_+^\ast \cup \vec{C}_-^\ast \} \cup \{(\{\vec{e},-\vec{e}\},\emptyset), (\emptyset, \{\vec{e},-\vec{e}\}): e \text{ is not a coloop} \}.\]
			\item The set of bases of $\Lambda(\M)$ is 
			\[\mathcal{B}(\Lambda(\M)) \coloneqq \{_A(E\backslash B) \cup B \cup \overline{B}:B \in \mathcal{B}(\M), A  \subseteq E \backslash B\}.\]
		\end{enumerate}
	\end{lem}
	
	Through the relationship between circuits of $\Lambda(\M)$, the circuits of $\M$, and the cocircuits of $\M^\ast$, we can identify extensions of $\M^\ast$ with subdivisions of $\Lambda(\M)$ using the following lemmas. 
	
	\begin{lem}[{\cite[Lemma 4.13(i)-(iii)]{triangulations-oriented}}]\label{lem:lawrence_polytopes_lifting_subdivisions}
		Let $\Lambda(\M)$ be the Lawrence polytope associated to an oriented matroid $\M$. Let $\Su$ be a subdivision of $\Lambda(\M)$. 
		\begin{enumerate}
			\item The support of every signed circuit of $\Lambda(\M)$ is a face of $\Lambda(\M)$. 
			\item Let $F$ be a face of $\Lambda(\M)$ which is the support of a circuit $\vec{C}$. Let $k$ be the rank of $F$. Consider the ``restriction" of $\Su$ to $F$ defined as:
			\[\Su_F \coloneqq \{\gamma \cap F:\gamma \in \Su, \text{rank}(\gamma \cap F) = k\}.\]
			Then, $\Su_F$ is either the trivial subdivision $\{F\}$ of the restricted oriented matroid $\Lambda(\M)(F) \coloneqq \Lambda(\M)\backslash F^c$ or one of the triangulations $\mathcal{T}_{\vec{C}}^+ \coloneqq \{(\vec{C_+}\cup\vec{C_-})\backslash \{e\}:e\in \vec{C_+}\}$ or $\mathcal{T}_{\vec{C}}^- \coloneqq \{(\vec{C_+}\cup\vec{C_-})\backslash \{e\}:e\in \vec{C_-}\}$ of the face $F$. 
			\item The signature of $\Lambda(\M)^\ast$ defined by $\vec{C}(p) = 0$ if $\Su_C = \{C\}$ and $\vec{C}(p) = +$ (resp. $\vec{C}(p)=-$) if $\Su_C = \mathcal{T}_{\vec{C}}^+$ (resp. $\Su_C = \mathcal{T}_{\vec{C}}^-$) is the signature of an extension $\Lambda(\M)^\ast \cup p$ of $\Lambda(\M)^\ast$.
		\end{enumerate}
	\end{lem}
	We note that in Lemma \ref{lem:lawrence_polytopes_lifting_subdivisions}(2), the triangulations $\mathcal{T}^+_{\vec{C}}$ and $\mathcal{T}^-_{\vec{C}}$ are triangulations of the oriented matroid $\Lambda(\M)(F)$, as defined in Chapter 2 of \cite{triangulations-oriented}.  
	
	\begin{thm}[{\cite[Theorem 4.14]{triangulations-oriented}}]\label{thm::subdivisions<->extensions}
		Let $\Lambda(\M)$ be the Lawrence polytope associated to an oriented matroid $\M$. Then, there is a natural bijection between the extensions of $\M^\ast$ and the subdivisions of $\Lambda(\M)$. 
	\end{thm}
	
	\section{Stability conditions and degeneracy sets} \label{sec:stabilityconditions}
	Let $X$ be a nodal curve, and $G$ be its dual graph. While stability conditions can be defined on the curve itself, we work with the dual graph $G$ throughout. 
	
	\begin{definition}[{\cite[Definition 2.1]{fava-pagani-viviani2024}}] \label{def:stability conditon}
		Let $\text{BCon(G)}$ denote the set of all biconnected and nontrivial $W \subset V(G)$. A \textit{$V-$stability condition} (or \textit{stability condition}) $\mathfrak{n}$ of degree $d \in \mathbb{Z}$ on $G$ is an assignment of integers 
		\begin{align*}
			\mathfrak{n}: \text{BCon(G)} &\to \mathbb{Z} \\
			W &\mapsto \mathfrak{n}_W
		\end{align*}
		satisfying the following properties: For any $W \in \text{BCon}(G)$, we have 
		\[\mathfrak{n}_W + \mathfrak{n}_{W^c} + |E(W, W^c)| - d \in \{0,1\},\]
		where any $W$ such that $\mathfrak{n}_W + \mathfrak{n}_{W^c} + |E(W, W^c)| - d = 0$ are called $\mathfrak{n}$-\textit{degenerate}, 
		and for any pairwise disjoint $W_1, W_2,$ and $W_3$ in $\text{BCon(G)}$ such that $W_1 \cup W_2 \cup W_3 = V(G)$, the following two properties hold:
		\begin{enumerate}
			\item for $i \neq j \neq k$, if $W_i$ and $W_j$ are $\mathfrak{n}$-degenerate, then $W_k$ is $\mathfrak{n}$-degenerate;
			\item the following equation is satisfied:
			\begin{equation}\label{eq:tribond-stabilitycond}
				\sum_{i=1}^3 \mathfrak{n}_{W_i} + \sum_{1 \leq i < j \leq 3} |E(W_i, W_j)| - d = \begin{cases}
					\{1,2\} & W_i \text{ is } \mathfrak{n}\text{-nondegenerate for all } i \\
					\{1\} & \exists! i \text{ such that } W_i \text{ is } \mathfrak{n}\text{-degenerate} \\
					\{0\} & W_i \text{ is } \mathfrak{n}\text{-degenerate for all } i.
				\end{cases}
			\end{equation}
		\end{enumerate}
		
		The \textit{degeneracy set} of a stability condition $\mathfrak{n}$, denoted $\mathcal{D}(\mathfrak{n})$, is the set of all $W \in \text{BCon}(G)$ such that $W$ is $\mathfrak{n}$-degenerate. The degree of the stability condition, $d$, is also denoted by $|\mathfrak{n}|$.
	\end{definition}
	
	One can use stability conditions to construct smoothable compactified Jacobians of $X$, as is done by Pagani and Tommasi \cite{pagani2023stability}; Viviani \cite{viviani2023new}; and by Fava, Pagani, and Viviani for nodal curves and reduced curves \cite{fava-pagani-viviani2024, fava-pagani-viviani2025}. Smoothable compactified Jacobians constructed from stability conditions include those constructed previously by Oda and Seshadri \cite{oda-seshadri}, Simpson \cite{simpson}, Caporaso \cite{caporaso}, and Esteves \cite{esteves}, and the compactifications corresponding to these previous constructions are called \textit{classical compactified Jacobians}. In this paper, we work with the graph theoretic construction given in \cite{oda-seshadri}. A \textit{numerical polarization} $\phi: V(G) \to \mathbb{R}$ is a real-valued divisor on a graph $G$ such that the degree $\deg(\phi) = \sum_{v \in V(G)} \phi(v)$ is an integer. For $W \in \text{BCon}(G)$, we let $\phi_W = \sum_{v \in W} \phi(v)$. Given a numerical polarization $\phi$, we can construct the $V-$stability condition $\mathfrak{n}(\phi)$, where each $\mathfrak{n}(\phi)_W = \left\lceil \phi_W - \frac{|E(W,W^c)|}{2} \right \rceil$. We say such a $V-$stability condition is \textit{classical}. 
	
	Let $V^d(G)$ be the affine space of numerical polarizations of degree $d$ supported on $V(G)$. One can construct an infinite arrangement of hyperplanes
	\[\mathcal{A}_G^d \coloneqq \left\{\phi_W - \frac{|E(W,W^c)|}{2} = n \right\}_{W \in \text{BCon}(G), n \in \mathbb{Z}},\]
	in $V^d(G)$, as is done in \cite{oda-seshadri} (see also \cite{pagani2023stability} and \cite{viviani2023new}). This yields a wall and chamber decomposition of $V^d(G)$. A numerical polarization $\phi \in V^d(G)$ is generic (or nondegenerate) if $\phi$ does not lie on any wall, and from the construction of $\mathfrak{n}(\phi)$, two generic polarizations $\phi$ and $\phi'$ are in the same chamber if and only if $\mathfrak{n}(\phi) = \mathfrak{n}(\phi')$.	
	
	Above, degeneracy sets are defined with respect to a stability condition, but a degeneracy set can also be defined without a reference stability condition. We recall the following definitions from \cite{fava-pagani-viviani2024}.
	
	\begin{definition}\label{def:deg-set}
		A \textit{degeneracy set} of a graph $G$ is a subset $\mathcal{D} \subseteq \text{BCon}(G)$ satisfying the following properties:
		\begin{enumerate}
			\item If $W \in \mathcal{D}$, then $W^c \in \mathcal{D}$;
			\item for all $W_1, W_2 \in \text{BCon}(G)$ such that $W_1 \cap W_2 = \emptyset$ and $W_1 \cup W_2 \in \text{BCon}(G)$, if $W_1, W_2 \in \mathcal{D}$, then $W_1 \cup W_2 \in \mathcal{D}$. 
		\end{enumerate}
	\end{definition}
	
	\begin{definition}\label{def:ext-deg-set}
		Let $\mathcal{D}$ be a degeneracy set for a graph $G$, and let $\text{Con}(G)$ denote all connected $W \subset V(G)$. The \textit{extended degeneracy subset} associated to $\mathcal{D}$ is the subset $\hat{\mathcal{D}} \subseteq \text{Con}(G)$ defined as
		\[\hat{\mathcal{D}} \coloneqq \{V \in \text{Con}(G):G[V^c] = \bigsqcup_{i\in I} G[W_i] \text{ with all } W_i \in \mathcal{D} \}.\]
	\end{definition}
	
	For a degeneracy set $\mathcal{D}$, a connected subgraph $G'$ of $G$ is \textit{$\mathcal{D}$-admissible} if $V(G') \in \hat{\mathcal{D}}$. If $G'$ is a spanning subgraph, but not necessarily connected, it admits a decomposition $G' = \bigsqcup_{i \in I}G'_i$ into connected components. A spanning subgraph $G'$ is \textit{$\mathcal{D}$-admissible} if $V(G_i') \in \hat{\mathcal{D}}$ for all $i \in I$. The poset of $\mathcal{D}$-admissible spanning subgraphs of $G$, ordered by inclusion, is denoted by $\mathcal{SS}_{\mathcal{D}}(G)$. We are especially interested in the poset of $\mathcal{D}$-admissible spanning forests of $G$, denoted by $\mathcal{SF}_{\mathcal{D}}(G)$, ordered by inclusion. The \textit{$\mathcal{D}$-complexity} of $G$ is $c_{\mathcal{D}}(G) = |\mathcal{SF}_{\mathcal{D}}(G)|$. Finally, we denote the minimal elements of $\mathcal{SS}_{\mathcal{D}}(G)$ by $m\mathcal{SS}_{\mathcal{D}}(G)$, and note that $m\mathcal{SS}_{\mathcal{D}}(G) \subseteq \mathcal{SF}_{\mathcal{D}}(G)$.
	
	One can restrict degeneracy sets and extended degeneracy sets of a graph $G$ to $\mathcal{D}$-admissible connected subgraphs of $G$. 
	
	\begin{lem}[{\cite[Lemma 3.5(2)]{fava-pagani-viviani2024}}] \label{lem:restrict-deg-set}
		Let $\mathcal{D}$ be a degeneracy set of $G$, $\hat{\mathcal{D}}$ be the extended degeneracy subset of $\mathcal{D}$, and let $G'$ be a $\mathcal{D}$-admissible connected subgraph of $G$. The restriction of $\mathcal{D}$ to $G'$ is 
		$\mathcal{D}(G') \coloneqq \{Z \in \text{BCon}(G'):Z \in \hat{\mathcal{D}}\}$ and its associated degeneracy set is $\hat{\mathcal{D}}(G') \coloneqq \{Z \in \text{Con}(G'):Z \in \hat{\mathcal{D}}\}$. 
	\end{lem}
	
	One can then connect stability conditions on $G$ to certain collections of divisors on $G$. For a forest $F$, let $b_0(F)$ denote the number of connected components of $F$.
	\begin{definition}
		Let $\mathcal{D}$ be a degeneracy set for a graph $G$. Let $I: m\mathcal{SS}_{\mathcal{D}}(G) \to \text{Div}(G)$ be any function that assigns a degree $d-g(G)-b_0(F)+1$ divisor $I(F)$ to every $F \in m\mathcal{SS}_{\mathcal{D}}(G)$, and call such an $I$ a \textit{$\mathcal{D}$-forest function of degree $d$.}
	\end{definition}
	
	For any spanning forest $F$ of $G$, let $\Or_{\mathcal{E}_F}$ be a choice of orientation of the edges in $\mathcal{E}_F \coloneqq E(G) \backslash E(F)$. By dropping a chip at the head of each oriented edge in $\Or_{\mathcal{E}_F}$, we get a divisor, which we denote by $\mathcal{D}(F, \Or_{\mathcal{E}_F})$.

	\begin{definition} 
		For a $\mathcal{D}$-forest function $I$, the set of \textit{generalized break divisors with respect to the $\mathcal{D}$-forest function $I$} is defined as 
		\[\mathcal{BD}_I(G) \coloneqq \{\mathcal{D}(F, \Or_{\mathcal{E}_F}) + I(F): F \in m\mathcal{SS}_{\mathcal{D}}(G) \text{ and } \Or_{\mathcal{E}_F} \text{ is an orientation of edges } \mathcal{E}_F\}.\]
	\end{definition}
	
	We can then view stability conditions in bijection with $\mathcal{D}$-forest functions that yield sets of generalized break divisors of a certain size.
	
	\begin{thm}[{\cite[Theorem 3.32(1)]{fava-pagani-viviani2024}}]\label{thm:stability-forest-fn-bij}
		Let $G$ be a connected graph. For any degeneracy set $\mathcal{D}$ of $G$, the map 
		\[\{V-\text{stability conditions } \mathfrak{n} \text{ on }G\} \leftrightarrow \{\mathcal{BD}_I(G) \text{ such that } |\mathcal{BD}_I(G)| = c_{\mathcal{D}(\mathfrak{n})}(G)\}\]
		is a bijection, in the following way: every V-stability condition $\mathfrak{n}$ with degeneracy set $\mathcal{D}(\mathfrak{n})$ corresponds to a $\mathcal{D}(\mathfrak{n})$-forest function $I$ such that $|\mathcal{BD}_I(G)| = c_{\mathcal{D}(\mathfrak{n})}(G)$.
	\end{thm}
	
	For future reference, we provide the constructions of a $\mathcal{D}$-forest function from a stability condition $\mathfrak{n}$ and of the stability condition $\mathfrak{n}^I$ from a $\mathcal{D}$-forest function $I$ satisfying $|\mathcal{BD}_I(G)| = c_{\mathcal{D}}(G)$. First, every stability condition $\mathfrak{n}$ can be extended to all subsets of $V(G)$ in the following way.
	
	\begin{definition}\label{def:extended-V-fn}
		The \textit{extended V-function} associated to a stability condition $\mathfrak{n}$ of degree $d$ is defined as follows:
		\begin{enumerate}
			\item If $W \subseteq V(G)$ is connected and $G[W^c] = \bigsqcup_{i=1}^k G[Z_i]$ is the decomposition of $G[W^c]$ into connected components, we define
			\[\mathfrak{n}_W = d - \sum_{i=1}^k \mathfrak{n}_{Z_i} - |E(W, W^c)| + |\{Z_i : Z_i \notin \mathcal{D}(\mathfrak{n})\}|.\]
			\item If $W \subseteq V(G)$ is not connected, and $\bigsqcup_{j=1}^h G[W_j]$ is the decomposition of $G[W]$ into connected components, we set 
			\[\mathfrak{n}_W = \sum_{j=1}^h \mathfrak{n}_{W_j}.\]
		\end{enumerate}
	\end{definition}
	
	We can also restrict a stability condition on a graph $G$ to a stability condition on a connected subgraph $G'$ of $G$. 
	
	\begin{lem}[{\cite[Lemma-Definition 2.16]{fava-pagani-viviani2024}}]\label{lem:restriction-n-to-subgraphs}
		Let $\mathfrak{n}$ be a stability condition of degree $d$ on a connected graph $G$. Let $G' = (V', E')$ be a connected, $\mathcal{D}(\mathfrak{n})$-admissible subgraph of $G$. We can write $G' = G[V']\backslash X$ for some $X \subseteq E(G)$. 
		
		Then, the restriction of $\mathfrak{n}$ to $G'$ is a stability condition $\mathfrak{n}(G')$ on $G'$ defined by 
		\[\mathfrak{n}(G')_W \coloneqq \mathfrak{n}_W - |X \cap E(G[W])|\]
		for any $W \in \text{BCon}(G')$.
		
		We then have the following:
		\begin{enumerate}
			\item $|\mathfrak{n}(G')| = \mathfrak{n}_{V(G')} - |X|$;
			\item $\hat{\mathcal{D}}(\mathfrak{n}(G')) = \hat{\mathcal{D}}(\mathfrak{n}) \cap  \text{Con}(G')$.
		\end{enumerate}
	\end{lem}
	
	We use the previous lemma to construct $\mathcal{D}$-forest functions from stability conditions $\mathfrak{n}$.
	
	\begin{lem}[{\cite[Lemma-Definition 3.30]{fava-pagani-viviani2024}}] \label{lem:forest-fn-from-n}
		Let $\mathfrak{n}$ be a stability condition of degree $d$ on $G$. Consider $F \in m\mathcal{SS}_{\mathcal{D}(\mathfrak{n})}(G)$ and let $F = \bigsqcup_{i=1}^k F_i$ be its decomposition into connected components. For any edge $e \in F$, let $F_{i_e}$ denote the connected component of $F$ containing $e$, and consider the decomposition of $F\backslash \{e\}$ given by
		\[F\backslash\{e\} = F[W_e] \cup F[V(F_{i_e})\backslash W_e] \bigsqcup_{i\neq i_e} F_i,\]
		where $F[W_e]$ and $ F[V(F_{i_e})\backslash W_e]$ are both connected. The $\mathcal{D}(\mathfrak{n})$-forest function of degree $d$ associated to $\mathfrak{n}$ is the function
		\[I_{\mathfrak{n}} :m\mathcal{SS}_{\mathcal{D}}(G) \to \text{Div}(G),\]
		where $I_{\mathfrak{n}}(F)$ is the unique divisor on $G$ satisfying
		\begin{equation}\label{eqn::I from n}
			\begin{cases}
				I_{\mathfrak{n}}(F)_{W_e} = \mathfrak{n}(F_{i_e})_{W_e} \\
				I_{\mathfrak{n}}(F)_{V(F_{i_e})\backslash W_e} = \mathfrak{n}(F_{i_e})_{V(F_{i_e}) \backslash W_e} \\
				I_{\mathfrak{n}}(F)_{V(F_i)} = |\mathfrak{n}(F_i)|.
			\end{cases}
		\end{equation}
	\end{lem}
	
	Now, given a $\mathcal{D}$-forest function $I$ satisfying $|\mathcal{BD}_I(G)| = c_{\mathcal{D}}(G)$, one can construct the corresponding stability condition $\mathfrak{n}^I$.
	\begin{lem}[{\cite[Lemma-Definition 3.31]{fava-pagani-viviani2024}}] \label{lem:forest-fn-to-n}
		Let $\mathcal{D}$ be a degeneracy set of a graph $G$, and let $I$ be a $\mathcal{D}$-forest function of degree $d$ such that $|\mathcal{BD}_I(G)| = c_{\mathcal{D}}(G)$. For any $W \in \text{BCon}(G)$, pick a minimal spanning forest $F \in m\mathcal{SS}_{\mathcal{D}}(G)$ that satisfies
		\begin{equation}\label{eqn:adapted-forest}
			\begin{cases}
				F[W] \text{ is a spanning forest of } G[W] \\
				F[W^c] \text{ is a spanning forest of } G[W^c] \\
				|E(F) \cap E(W,W^c)| = 1 - \mathbf{1}_{\mathcal{D}}(W),
			\end{cases}
		\end{equation}
		where $\mathbf{1}_{\mathcal{D}}(W)$ is $1$ if $W \in \mathcal{D}$ and $0$ otherwise. Then 
		\begin{align*}
			\mathfrak{n}^I: \text{BCon}(G) &\to \mathbb{Z} \\
			W &\mapsto \mathfrak{n}_W^I \coloneqq I(F)_W + g(G[W]) + b_0(F[W]) - 1
		\end{align*}
		is a stability condition of degree $d$ on $G$ with $\mathcal{D}(\mathfrak{n}^I) = \mathcal{D}$.
	\end{lem}
	Finally, we are interested in $\mathcal{D}$-forest functions $I$ such that every $I(F)$ is an orientable divisor. We call such an $I$ an \textit{orientable} $\mathcal{D}$-forest function.
	
	\begin{lem}\label{lem:I-orientability}
		Let $I$ be a $\mathcal{D}$-forest function for some degeneracy set $\mathcal{D}$ of $G$ satisfying $|\mathcal{BD}_I(G)| = c_{\mathcal{D}}(G)$. Then, $I$ is an orientable $\mathcal{D}$-forest function if and only if
		\begin{equation} \label{eqn:forest-cases}
			\begin{cases}
				I(F)_{W_e} \in \{0, -1\} \\
				I(F)_{V(F_{i_e})\backslash W_e} \in \{0, -1\} \backslash I(F)_{W_e} \\
				I(F)_{V(F_i)} = -1.
			\end{cases}
		\end{equation}
	\end{lem}
	\begin{proof}
		If $I$ is an orientable $\mathcal{D}$-forest function, then $I(F)$ is an orientable divisor for every $F \in m\mathcal{SS}_{\mathcal{D}}(G)$. Let $\Or$ be the unique orientation of $F$ such that $D_\Or = I(F)$. We note that $\Or$ is unique because $F$ is a forest. Let $\sqcup_i F_i$ be the decomposition of $F$ into connected components. Since $\Or$ orients each $F_i$, we have $I(F_i) = -1$ for each $F_i$. For $F_{i_e}$, if we contract all the edges except for $e$ in $F[W_e]$, $e$ is either oriented away from $W_e$ or towards $W_e$, so $I(F)_{W_e}$ is $0$ or $-1$. Because $I(F_{i_e}) = -1$, we must have $I(F)_{V(F_{i_e})\backslash W_e} \in \{0, -1\} \backslash I(F)_{W_e}$, as desired.
		
		Now, assume $I(F)$ satisfies Equation \ref{eqn:forest-cases}. Then, we construct an orientation $\Or$ of $F$ in the following natural way: consider an edge $e$ of $F$, and let $F_{i_e}$ denote the connected component of $F$ containing $e$. Then, if $I(F)_{W_e} = 0$, we orient $e$ towards $W_e$; otherwise, we orient $e$ away from $W_e$. By construction, $D_\Or = I(F)$, so $I$ is orientable. 
	\end{proof}
	
	\begin{corollary}\label{cor:stab-degree-from-orientable-I}
		Let $I$ be an orientable $\mathcal{D}$-forest function satisfying $|\mathcal{BD}_I(G)| = c_{\mathcal{D}}(G)$. Then, $I$ is a $\mathcal{D}$-forest function of degree $g(G) - 1$. 
	\end{corollary}
	
	\begin{corollary} \label{cor:stab-values-from-orientable-I}
		Let $I$ be an orientable $\mathcal{D}$-forest function satisfying $|\mathcal{BD}_I(G)| = c_{\mathcal{D}}(G)$. Then, the corresponding stability condition $\mathfrak{n}^I$ satisfies
		\begin{equation}\label{eqn:orientable-n-values}
			\mathfrak{n}^I_W = \begin{cases}
				g(G[W]) - 1 & W \in \mathcal{D} \\
				g(G[W]) & W \notin \mathcal{D} \text{ and } I(F)_W = 0 \text{ for all forests satisfying Equation \ref{eqn:adapted-forest}} \\
				g(G[W])-1 & W \notin \mathcal{D} \text{ and } I(F)_W = -1 \text{ for all forests satisfying Equation \ref{eqn:adapted-forest}}. 	
			\end{cases}
		\end{equation}
	\end{corollary}
	\begin{proof}
		Let $F$ be a minimal spanning forest in $m\mathcal{SS}_\mathcal{D}(G)$ satisfying Equation \ref{eqn:adapted-forest}. If $W \in \mathcal{D}$, then $|E(F) \cap E(W, W^c)| = 0$, so $F[W]$ is the union of disconnected components $F_i$ of $F$. For each component $F_i$, we have $I(F_i) = -1$ from Lemma \ref{lem:I-orientability}, so $I(F)_W = -b_0(F[W])$ and the result follows from Lemma \ref{lem:forest-fn-to-n}. Now, consider if $W \notin \mathcal{D}$, so $|E(F) \cap E(W,W^c)| = 1$. Let $e$ be the edge in the intersection $E(F) \cap E(W, W^c)$, and $F_{i_e}$ be the connected component of $F$ containing $e$. Then, $F[W]$ is the disjoint union of connected components $F_i$ not containing $e$ and $F_{i_e}[V(F_{i_e}) \cap W]$. From Lemma \ref{lem:I-orientability}, we have $I(F_i) = -1$ for each $F_i$ and $I(F_{i_e}) = -1$. If $e$ was oriented towards $F_{i_e}[V(F_{i_e}) \cap W]$, then $I(F)_W = 1 - b_0(F[W])$; otherwise, $I(F)_W = - b_0(F[W])$. 
	\end{proof}
	
	\subsection{Poset structures on stability conditions and degeneracy sets}
	We are also interested in two posets relating to stability conditions. 
	\begin{definition}[{\cite[Definition 4.11]{fava-pagani-viviani2025}}] \label{def:poset-relation-stab-set}
		Let $\text{VStab}(G)$ denote the poset of all stability conditions on $G$ with the following relation: for $\mathfrak{n}^1$ and $\mathfrak{n}^2$ in $\text{VStab}(G)$, $\mathfrak{n}^1 \geq \mathfrak{n}^2$ if $|\mathfrak{n}^1| = |\mathfrak{n}^2|$ and $\mathfrak{n}^1_W \geq \mathfrak{n}^2_W$ for all $W \in \text{BCon}(G)$. 
	\end{definition}
	
	\begin{definition}[{\cite[Definition 4.18]{fava-pagani-viviani2025}}] \label{def:poset-relation-deg-set}
		Let $\text{Deg}(G)$ denote the poset of all degeneracy sets of $G$, with the following relation: for $\mathcal{D}_1, \mathcal{D}_2 \in \text{Deg}(G)$, we say that $\mathcal{D}_1 \geq \mathcal{D}_2$ if $\mathcal{D}_1 \subseteq \mathcal{D}_2$ and there exists a subset $\mathcal{E} \subset \mathcal{D}_2 \backslash \mathcal{D}_1$ such that the following conditions hold:
		\begin{enumerate}
			\item $\mathcal{D}_2 \backslash \mathcal{D}_1 = \mathcal{E}\bigsqcup\mathcal{E}^c$, where $\mathcal{E}^c = \{Y^c:Y \in \mathcal{E}\}$,
			\item for any $Z_1, Z_2 \in \mathcal{D}_2 \backslash \mathcal{D}_1$ such that $Z_1 \cap Z_2 = \emptyset$ and $Z_1 \cup Z_2 \in \mathcal{D}_1$, either $Z_1$ or $Z_2$ is in $\mathcal{E}$, but not both. 
			\item for any $Z_1, Z_2, Z_3 \in \mathcal{D}_2 \backslash \mathcal{D}_1$ such that $Z_1, Z_2, Z_3$ are all pairwise disjoint, and $V(G) = Z_1 \cup Z_2 \cup Z_3$, $\mathcal{E}$ contains exactly $1$ or $2$ of the $Z_i$. 
		\end{enumerate}
	\end{definition}
	These posets are compatible with each other in the following way. We can consider $\mathcal{D}$ as a map from stability conditions $\mathfrak{n}$ to their degeneracy sets $\mathcal{D}(\mathfrak{n})$.
	\begin{proposition}[{\cite[Proposition 4.19]{fava-pagani-viviani2025}}]\label{prop:order-preserving-upper-lifting}
		The map $\mathcal{D}$ satisfies the following properties:
		\begin{enumerate}
			\item $\mathcal{D}$ is order-preserving: if $\mathfrak{n}^1 \geq \mathfrak{n}^2$, then $\mathcal{D}(\mathfrak{n}^1) \geq \mathcal{D}(\mathfrak{n}^1)$.
			\item $\mathcal{D}$ is upper lifting: for any $\mathcal{D}_1 \geq \mathcal{D}_2$ in $\text{Deg}(G)$, and for any stability condition $\mathfrak{n}^2$ such that $\mathcal{D}(\mathfrak{n}^2) = \mathcal{D}_2$, there exists a stability condition $\mathfrak{n}^1$ such that $\mathfrak{n}^1 \geq \mathfrak{n}^2$ and $\mathcal{D}(\mathfrak{n}^1) = \mathcal{D}_1$. 
		\end{enumerate}
	\end{proposition}
	
	The stability condition $\mathfrak{n}^1$ in Proposition \ref{prop:order-preserving-upper-lifting}(2) can be constructed explicitly: let $\mathcal{E} \subset \mathcal{D}_2 \backslash \mathcal{D}_1$ be the set satisfying the conditions in Definition \ref{def:poset-relation-deg-set}. Then, $\mathfrak{n}^1$ is defined by
	\begin{equation}\label{eqn:upper-lifting-stab-condition}
		\mathfrak{n}^1_W = \begin{cases}
			\mathfrak{n}^2_W & W \in \mathcal{D}_1 \text{ or } W \notin \mathcal{D}_2 \\
			\mathfrak{n}^2_W + 1 & W \in \mathcal{E} \\
			\mathfrak{n}^2_W & W \in \mathcal{E}^c.
		\end{cases}
	\end{equation}

	\subsection{Triangulations and nondegenerate stability conditions}
	In \cite{crepeau}, we were motivated by studying functions $I: \mathcal{ST}(G) \to \text{Div}^0(G)$ such that $|\mathcal{BD}_I(G)| = |\mathcal{ST}(G)| = c_\emptyset(G)$; in this case, $I$ is a degree $g(G)$ $\mathcal{D}$-forest function for $\mathcal{D}=\emptyset$. To do this, we used a family of bijections constructed by Ding \cite{ding2023framework} between the bases of a regular oriented matroid $\M$ and representatives of the circuit-cocircuit reversal system of $\M$. For an oriented matroid $\M$ and a basis $B$ of $\M$, an \textit{internally oriented basis} $\vec{B^\ast}$ is a fourientation of $\M$ where all the edges of $B$ are one-way oriented, and all the edges not in $B$ are bioriented. An \textit{internal atlas} of $\M$ is a collection of internally oriented bases $\vec{B}^\ast$ where each basis of $\M$ appears once. The internal atlases of interest in this paper are those that can be constructed from triangulations of the Lawrence polytope $\Lambda(\M^\ast)$; such an atlas is called \textit{triangulating} and can be combinatorially characterized without a reference triangulation.
	
	\begin{thm}[{\cite[Theorem 1.30]{ding2023framework}}]\label{thm: triang-to-atlas}
		Let $\mathcal{P}^\ast$ be the Lawrence polytope of the matroid $\M^\ast$ dual to the loopless regular matroid $\M$. The map
		\[\chi^\ast: \{\text{triangulations } \T \text{ of } \mathcal{P}^\ast\} \to \{\text{triangulating internal atlases of } \M\}\]
		is a bijection, where for a triangulation $\T$ of $\mathcal{P}^\ast$, each simplex $\tau \in \T$ corresponds to an internally oriented basis of $\M$ in the following way: $\tau$ is the convex hull of $\{P_j:j\in J\}$ for some $J \subseteq \{1,...,n,-1,...,-n\}$, and $\chi^\ast(\tau)$ is the fourientation given by $\{\vec{e_j}:j\in J\}$.
	\end{thm}
	
	For the graphic matroid $\M(G)$, $\chi^\ast$ maps triangulations of $\Lambda(\M^\ast(G))$ to internally oriented spanning trees $\vec{T^\ast}$ of $G$. We recall that if $\vec{F}$ is a fourientation, then $-(\vec{F})^c$ switches bioriented edges with unoriented ones, so for an internal orientation $\vec{B^\ast}$, $(-\vec{B^\ast})^c$ is an orientation of $B$. Using $\chi^\ast$, we constructed functions $I$ such that $|\mathcal{BD}_I(G)| = |\mathcal{ST}(G)|$.
	
	\begin{thm}[{\cite[Theorem 5.1]{crepeau}}]\label{thm:triang-to-goodIG}
		Let $\mathcal{T}$ be a triangulation of the Lawrence polytope $\Lambda(\M^\ast(G))$. Then, the function
		\begin{align*}
			I^\T:\mathcal{ST}(G) &\to \text{Div}^{-1}(G) \\
			T &\mapsto D_{-(\vec{T^\ast})^c}
		\end{align*}
		satisfies $|\mathcal{BD}_{I^\T}(G)| = |\mathcal{ST}(G)|$. 
	\end{thm}
	
	By Theorem \ref{thm:stability-forest-fn-bij}, $I^\T$ corresponds to a nondegenerate stability condition $\mathfrak{n}^\T$ of degree $g(G)-1$. Another motivating question was to understand which stability conditions are induced by a generic numerical polarization. 
	
	\begin{thm}[{\cite[Theorem 5.10]{crepeau}}] \label{thm:acyclic-classical}
		If $\mathcal{T}$ is a regular triangulation of the Lawrence polytope $\Lambda(\M^\ast(G))$, then the stability condition $\mathfrak{n}^\mathcal{T}$ corresponding to the function $I^\T$ is classical.
	\end{thm}

	\section{Subdivisions and stability conditions}
	
	\subsection{Single-element extensions of matroids and degeneracy sets}
	
	For a degeneracy set, it can be useful to consider the cocycles of $G$ corresponding the the biconnected subsets of vertices in $\mathcal{D}$. We define $C^\ast(\mathcal{D}) \coloneqq \{E(W,W^c):W\in \mathcal{D}\}$. 
	
	The following lemmas are known (see, e.g., \cite[Section 2]{slilaty}), but we state them in the language of this paper for reference.
	\begin{lem}\label{lem:modular-pairs}
		Let $W_1$ and $W_2$ be biconnected subsets of $V(G)$ such that $W_1 \cap W_2 = \emptyset$, and either $W_1 \cup W_2$ is biconnected and nontrivial, or $(W_1 \cup W_2)^c$ is connected and nontrivial. Then, the cocycles $C_1^\ast = E(W_1,W_1^c)$ and $C_2^\ast =E(W_2, W_2^c)$ are a modular pair of cocircuits of $M(G)$.
	\end{lem}
	\begin{proof}
		Let $r^\ast$ be the rank function of cographic matroid $M^\ast(G)$; for any subset $S \subseteq E$, we have $r^\ast(S) = r(E\backslash S) + |S| - r(M(G))$ (see, e.g., \cite[Proposition 2.1.9]{oxley}). We recall that $C_1^\ast$ and $C_2^\ast$ are a modular pair of cocircuits of $M(G)$ if they are a modular pair of circuits for $M^\ast(G)$. By duality, $C^\ast_1$ and $C_2^\ast$ are circuits of $M^\ast(G)$, so $r^\ast(C_1^\ast) = |C_1^\ast| - 1$ and $r^\ast(C_2^\ast) = |C_2^\ast|-1$. The intersection $C_1^\ast \cap C_2^\ast$ is the set of edges $E(W_1,W_2)$. Since $(W_1 \cup W_2)^c$ is connected, the subgraph $G \backslash (E(W_1, W_2))$ is connected, and $r(E\backslash E(W_1, W_2)) = |V(G)|-1$. Therefore, $r^\ast(C_1^\ast \cap C_2^\ast) = |E(W_1, W_2)|$. 
		
		Because the subsets $W_1$ and $W_2$ are disjoint and biconnected, $G\backslash (C_1^\ast \cup C_2^\ast)$ has $3$ connected components. Therefore, $r^\ast(C_1^\ast \cup C_2^\ast) = |C_1^\ast| + |C_2^\ast| - |E(W_1, W_2)| - 2$.
		
		We then see $r^\ast(C_1^\ast) + r^\ast(C_2^\ast) = r^\ast(C_1^\ast \cap C_2^\ast) + r^\ast(C_1^\ast \cup C_2^\ast)$, so $C_1^\ast$ and $C_2^\ast$ is a modular pair of cocircuits of $M(G)$. 
	\end{proof}
	
	\begin{lem}\label{lem:rank2minorexists}
		Consider biconnected, disjoint subsets $W_1$ and $W_2$ of $V(G)$ such that $W_1 \cup W_2$ is a biconnected, nontrivial subset of $V(G)$. Then, there exists a rank $2$ minor of $\M(G)$, which we denote $\mathcal{N}$, such that $\mathcal{C}^\ast(\mathcal{N}) = \{E(W_1, W_1^c), E(W_2, W_2^c), E(W_1 \cup W_2, (W_1 \cup W_2)^c)\}$. 
	\end{lem}
	\begin{proof}
		Let $E' = E(G[W_1]) \sqcup E(G[W_2]) \sqcup E(G[(W_1 \cup W_2)^c])$. Each of these sets is disjoint because $W_1$, $W_2$, and $(W_1 \cup W_2)^c$ are all disjoint from each other. The contraction $G/E'$ is a three vertex graph with edges $E(G/E') = E(W_1, W_1^c) \cup E(W_2, W_2^c)$; therefore, $\M(G)/E'$ is a rank $2$ minor of $\M(G)$. Identify the vertices of $G/E'$ as $W_1$, $W_2$, and $(W_1 \cup W_2)^c$. Then, the cocycles of $G/E'$ are exactly $E(W_1, W_1^c), E(W_2, W_2^c)$, and $ E(W_1 \cup W_2, (W_1 \cup W_2)^c)$.
	\end{proof}
	We recall from Subsection \ref{subsec:extension} that single-element extensions of unoriented matroids are described by functions, called signatures, on the cocircuits of the matroid. The requirements for a set $\mathcal{D}$ to be a degeneracy set of a graph $G$ can be translated to the requirements for a signature $\sigma^\ast$ to be a signature of a single-element extension of $M(G)$. 
	\begin{lem}\label{lem:degsets-subclasses}
		$\mathcal{D}$ is a degeneracy set of a graph $G$ if and only if there exists a single-element extension of $M$ by an element $p$ such that $\{C^\ast \in C^\ast(M):\sigma^\ast(C^\ast)=0\} = C^\ast(\mathcal{D})$. 
	\end{lem}
	\begin{proof}
		First, let $\mathcal{D}$ be a degeneracy set of a graph $G$. Then, we define a signature $\sigma^\ast$ such that $\sigma^\ast(C^\ast) = 0$ if and only if $C^\ast \in C^\ast(\mathcal{D})$. We then consider modular pairs $C_1^\ast, C_2^\ast$ of cocircuits such that $\sigma^\ast(C_1^\ast) = \sigma^\ast(C_2^\ast) = 0$. For such a pair, we have $C_1^\ast, C_2^\ast \in C^\ast(\mathcal{D})$. Let $W_1, W_2$ be the biconnected sets in $\mathcal{D}$ such that $C_1^\ast = E(W_1, W_1^c)$ and $C_2^\ast = E(W_2, W_2^c)$. We note that $W_1 \cap W_2 = \emptyset$ in order for $C_1^\ast, C_2^\ast$ to be a modular pair, and $W_1 \cup W_2$ is nontrivial since $C_1^\ast \neq C_2^\ast$. Then, consider $C_3^\ast \subset C_1^\ast \cup C_2^\ast$. By the minimality of cocircuits, $C_3^\ast = E(W_1 \cup W_2, (W_1\cup W_2)^c)$. Because $\mathcal{D}$ is a degeneracy set, $W_1 \cup W_2 \in \mathcal{D}$, and $\sigma^\ast(C_3^\ast) = 0$, by definition. By Theorem \ref{thm:sigs-unoriented extensions}, there exists a single-element extension of $M(G)$ such that $\sigma^\ast$ is the signature of that extension.
		
		Now, consider some extension $M'$ of $M(G)$ by an element $p$, with signature $\sigma^\ast$. Let $\mathcal{D}(M') = \{W, W^c:\sigma^\ast(E(W,W^c)) = 0\}.$ The first condition of Definition \ref{def:deg-set} is immediately satisfied, and by Theorem \ref{thm:sigs-unoriented extensions}(3), the second condition is satisfied. Therefore $\mathcal{D}(M')$ is a degeneracy set.
	\end{proof}
	
	\begin{corollary}
		If $\mathcal{D}$ is a degeneracy set for a graph $G$, then $C^\ast(\mathcal{D})$ is a linear subclass of cocircuits. 
	\end{corollary}
	
	A graphic matroid $M(G)$ may have single-element extensions $M'$ that are not orientable matroids, such as in Example \ref{ex:not-all-extensions-orientable}. We say a degeneracy set $\mathcal{D}$ is \textit{orientable} if the single-element extension corresponding to $\mathcal{D}$ is an orientable matroid.

	Recall that the set of all degeneracy sets of a graph $G$ is denoted by $\text{Deg}(G)$. While $\text{Deg(G)}$ and the set of linear subclasses of cocircuits of $M(G)$ are the same, the poset structure on $\text{Deg(G)}$ is stronger than inclusion. We let $\Or\text{Deg}(G)$ denote the set of all orientable degeneracy sets with the poset structure defined in Definition \ref{def:poset-relation-deg-set}. By definition, every orientable degeneracy set corresponds (non-uniquely) to a single-element extension of $\M(G)$, and therefore corresponds (non-uniquely) to a subdivision of $\Lambda(\M^\ast(G))$. For any subdivision $\Su$ of $\Lambda(\M^\ast(G))$, we then describe the degeneracy set $\mathcal{D}_\Su$. 
	\begin{proposition}
		Every subdivision $\Su$ of $\Lambda(\M^\ast(G))$ corresponds to some orientable degeneracy set $\mathcal{D}_S \subseteq \text{BCon}(G)$. 
	\end{proposition}
	\begin{proof}
		Let $\Su$ be a subdivision of $\Lambda(\M^\ast(G))$. By Theorem \ref{thm::subdivisions<->extensions}, the subdivision $\Su$ corresponds to a single-element extension of $\M(G)$. Let $\sigma^\ast$ be the localization describing that extension of $\M(G)$, as defined in Theorem \ref{thm::or-ext-equivalences}. Let $\vec{C}^\ast_W$ denote the cocircuit $E(W, W^c)$ oriented towards $W$, so $\vec{C}^\ast_W = -\vec{C}^\ast_{W^c}$. Consider the set $\mathcal{D}_{\Su} \coloneqq \{W \in \text{BCon(G)}: \sigma^\ast(\pm\vec{C}^\ast_W) = 0\}$. 
		
		We wish to show that $\mathcal{D}_{\Su}$ is a degeneracy set of the graph $G$. First, it is trivial to see that if $W \in \mathcal{D}_{\Su}$, then $W^c \in \mathcal{D}_{\Su}$. We then check the second condition of Definition \ref{def:deg-set}. Consider biconnected subsets $W_1$ and $W_2$ in $\mathcal{D}_S$ such that $W_1 \cap W_2 = \emptyset$, and $W_1 \cup W_2$ is biconnected and nontrivial. Note that if $G$ has less than 3 vertices, there does not exist such a $W_1$ and $W_2$, and this condition is trivially satisfied. We then assume $G$ has at least three vertices. Then, for the cocircuits $C_1^\ast = E(W_1, W_1^c)$ and $C_2^\ast = E(W_2, W_2^c)$, we have $\sigma^\ast(\vec{C}_1^\ast) = \sigma^\ast(\vec{C}_2^\ast) = 0$. Let $C_3^\ast = E(W_1 \cup W_2, (W_1\cup W_2)^c)$, which is also a cocircuit of $\M(G)$. By Lemma \ref{lem:rank2minorexists}, we can consider the rank two minor $\mathcal{N} = \M(G) / E'$, where $E' = E(G[W_1]) \sqcup E(G[W_2]) \sqcup E(G[(W_1 \cup W_2)^c])$. The minor $\mathcal{N}$ can be reduced to having three edges by deletion. Since $\sigma^\ast$ is a localization, by Theorem \ref{thm::or-ext-equivalences}, it cannot produce one of the three excluded subconfigurations in Figure \ref{fig:forbiddenconfigs}. If $\sigma^\ast(\vec{C}_3^\ast) = \pm$, we have the third forbidden subconfiguration in Figure \ref{fig:forbiddenconfigs}. Therefore, we must have $\sigma^\ast(\vec{C}_3^\ast) = \sigma^\ast(-\vec{C}_3^\ast) = 0$, and $W_1 \cup W_2 \in \mathcal{D}_{\Su}$, as desired.     
	\end{proof}
	We note that $\mathcal{D}_\Su$ is not unique to the subdivision; for example, every triangulation $\T$ of $\Lambda(\M^\ast(G))$ corresponds to $\mathcal{D}_\T = \emptyset$. 
	\subsection{Subdividing atlases}
	As is done for triangulations in Theorem \ref{thm: triang-to-atlas}, we wish to show that every subdivision of the Lawrence polytope $\Lambda(\M^\ast(G))$ corresponds to a collection of internally oriented minimal $\mathcal{D}_\Su$-admissible spanning forests of $G$, which we call a \textit{subdividing atlas}. We extend the map $\chi^\ast$ defined in Theorem \ref{thm: triang-to-atlas} to a map on subdivisions of Lawrence polytopes in the natural way.
	\begin{lem} \label{lem::cell->orientation}
		Let $\Su$ be some subdivision of $\Lambda(\M^\ast(G))$. Then, for each cell $c \in \Su$, we have that $\chi^\ast(c)$ is the internal orientation of some forest $F$ of $G$.
	\end{lem}
	\begin{proof}
		Recall that every subdivision can be refined to a triangulation of $\Lambda(\M^\ast(G))$. Say that triangulation is $\mathcal{T}$. Then, $c$ is the union of some simplices in $\mathcal{T}$, which each correspond to an internally oriented spanning tree. The union of such fourientations will be some internally oriented forest $F$ of $G$.
	\end{proof}
	
	\begin{lem}\label{lem:cells are admissible}
		Let $\Su$ be a subdivision of $\Lambda(\M^\ast(G))$. Then, for every cell $c \in \Su$, $\chi^\ast(c)$ is an internal orientation of a  $\mathcal{D}_{\Su}$-admissible spanning forest of $G$.
	\end{lem}
	
	\begin{proof}
		By Lemma \ref{lem::cell->orientation}, $\chi^\ast(c)$ is an internally oriented spanning forest $\vec{F^\ast}$ of $G$. The forest $F$ has a decomposition $F = \bigsqcup_{i\in I} F_i$ into connected components. If $F$ is connected, then $F$ is a spanning tree.
		
		Recall that $F$ is $\mathcal{D}_{\Su}$-admissible if $V(F_i) \in \hat{\mathcal{D}_{\Su}}$, where $\hat{\mathcal{D}_{\Su}}$ is the extended degeneracy set of $\hat{\mathcal{D}_\Su}$ defined in Definition \ref{def:ext-deg-set}. We consider $V(F_i)$. Every spanning forest $F$ of $G$ is a subgraph of a spanning tree of $G$. Recall that removing edges from a spanning tree $T$ partitions $V(G)$ into biconnected subsets $W$ and $W^c$. In this way, we can view $G[V(F_i)^c]$ as the disjoint union of $G[W_j]$ for $j \in J$, where $W_j$ is a biconnected subset of $V(G)$. Let $C_j^\ast = E(W_j, W_j^c)$. Thus, if we show that every $W_j$ is in $\mathcal{D}_{\Su}$, we have $V(F_i) \in \hat{\mathcal{D}}_{\Su}$. If $F$ is a spanning tree, then $V(F) = V(G) \in \hat{\mathcal{D}}_{\Su}$, as desired, so we now assume $F$ is not a spanning tree. 
		
		Let $F_{\vec{C}^\ast_j}$ denote the face of $\Lambda(\M^\ast(G))$ corresponding to the cocircuit $\vec{C}_j^\ast$ of $\M(G)$, as described in Lemma \ref{lem:lawrence_polytopes_lifting_subdivisions}(2). For every cocycle $C_j^\ast$ of $G$, the edges of $C_j^\ast$ are bioriented in $\vec{F^\ast}$. Therefore, the intersection of the cell $c$ with the face $F_{\vec{C}^\ast_j}$ is the face $F_{\vec{C}^\ast_j}$ for every $j \in J$. As a result, the induced subdivision $\Su_{F_{\vec{C}_j^\ast}}$ is the trivial subdivision for every $j \in J$. By Lemma \ref{lem:lawrence_polytopes_lifting_subdivisions}(3), the localization $\sigma^\ast_\Su$ corresponding to the subdivision $\Su$ satisfies $\sigma^\ast_\Su(\pm \vec{C}_j^\ast) = 0$ for all $j \in J$. By definition of $\mathcal{D}_{\Su}$, we have $W_j \in \mathcal{D}_{\Su}$ for every $j\in J$. Since every $V(F_i) \in \hat{\mathcal{D}}_{\Su}$, we have that $F$ is $\mathcal{D}_{\Su}$-admissible, as desired.
	\end{proof}
	
	\begin{lem}\label{lem::subdiv->oriented forests}
		Let $\Su$ be a subdivision of $\Lambda(\M^\ast(G))$. Then, for every cell $c \in \Su$, $\chi^\ast(c)$ is an internally orientated forest $\vec{F}^\ast$ for some $F \in m\mathcal{SS}_{\mathcal{D}_S}(G)$.
	\end{lem}
	
	\begin{proof}
		From Lemma \ref{lem:cells are admissible}, let $\underline{\chi^\ast(c)}$ denote the $\mathcal{D}_S$-admissible spanning forest that $\chi^\ast(c)$ is an internal orientation of. For contradiction, assume that $\underline{\chi^\ast(c)}$ is not minimal in the poset $\mathcal{SS}_{\mathcal{D}_\Su}(G)$. Therefore, there exists some minimal $\mathcal{D}_S$-admissible forest $H$ such that $\underline{\chi^\ast(c)} > H$. Additionally, let $T$ be a spanning tree in $\mathcal{SS}_{\mathcal{D}_\Su}(G)$ such that $T \geq \underline{\chi^\ast(c)} > H$. Such a spanning tree exists because all spanning trees are $\mathcal{D}$-admissible for any degeneracy set $\mathcal{D}$.
		
		Since $T$ is a connected $\mathcal{D}_\Su$-admissible subgraph of $G$, we can consider the restriction of $\mathcal{D}_\Su$ to $T$, as defined in Lemma \ref{lem:restrict-deg-set}. We have
		\[\mathcal{D}_\Su(T) = \{Z \in \text{BCon}(T):Z \in \hat{\mathcal{D}}_\Su\}.\]
		
		Since $\text{BCon}(T) \subseteq \text{BCon}(G)$, if $Z \in \text{BCon}(T)$ and $Z \in \hat{\mathcal{D}}_\Su$, $Z$ is also biconnected in $G$ and $Z \in \mathcal{D}_\Su$. Therefore, 
		\[\mathcal{D}_\Su(T) = \{Z \in \text{BCon}(T):Z \in \mathcal{D}_\Su\}.\]
		
		We note that both $\underline{\chi^\ast(c)}$ and $H$ are also $\mathcal{D}_\Su(T)$-admissible.
		
		Consider some edge $e$ in $\underline{\chi^\ast(c)}$ but not in $H$. Then, there exists a connected component of $H$, called $H_e$, such that $e$ is incident to some vertex in $H$. Since $H$ is $\mathcal{D}_\Su(T)$-admissible, we have $V(H_e) \in \hat{\mathcal{D}}_\Su(T)$, so $G[(V(H_e))^c] = \sqcup_{i \in I} G[W_i]$, where all $W_i$ are in $\mathcal{D}_\Su(T)$. There exists a unique $W_i$ in this decomposition such that $e \in E(W_i, W_i^c)$, and by definition $W_i \in \mathcal{D}_\Su(T)$, so $W_i \in \mathcal{D}_\Su$. Let $C_i^\ast$ be the cocycle $E(W_i, W_i^c)$ of $G$. Then, $C^\ast_i$ is the fundamental cocycle $C^\ast(T,e)$. Consider the face of the Lawrence polytope $\Lambda(\M^\ast(G))$ corresponding to the the cocircuit $\vec{C}^\ast_i$ of $\M(G)$, which we denote by $F_{\vec{C}_i^\ast}$. Since $\underline{\chi^\ast(c)}$ contains the edge $e$, we have that $\chi^\ast(c) \cap F_{\vec{C}_i^\ast}$ is a cell of the triangulation $\mathcal{T}_{\vec{C}_i^\ast}^+$ or $\mathcal{T}_{\vec{C}_i^\ast}^-$ of the restricted oriented matroid $\Lambda(\M^\ast(G))(F_{\vec{C}_i^\ast})$. Therefore, the subdivision $S$ orients the cocircuit $C^\ast$, and by definition $W_i \notin \mathcal{D}_S$, a contradiction.
	\end{proof}
	
	The following lemma translates Lemma \ref{lem:lawrence_polytopes_lifting_subdivisions} into the language of this paper.
	
	\begin{lem}\label{lem:fundcocircuits-restricted triangs}
		Consider a subdivision $\Su$ of $\Lambda(\M^\ast(G))$. Consider a signed cocircuit $\vec{C}^\ast$ of $\M(G)$, and let $F_{\vec{C}^\ast}$ be the face of $\Lambda(\M^\ast(G))$ corresponding to $\vec{C}^\ast$. Let $c$ be a cell of $\Su$. Then, $c \cap F_{\vec{C}^\ast}$ is a nontrivial cell of the restriction of $\Su$ to $F_{\vec{C}^\ast}$ if and only if  $\vec{C}^\ast \subseteq \chi^\ast(c)$, $C^\ast$ is a fundamental cocycle of a spanning tree $T$ such that $T \geq_{\Su\Su_{\mathcal{D}_\Su}(G)} \underline{\chi^\ast(c)}$, and $\sigma_\Su^\ast(\vec{C}^\ast) = \pm$. 
	\end{lem}
	\begin{proof}
		Let $\T$ be a triangulation of $\Lambda(\M^\ast(G))$ that refines $\Su$. For the reverse direction, by Lemma \ref{lem::subdiv->oriented forests}, $\chi^\ast(c)$ is an internal orientation of a forest $\underline{\chi^\ast(c)} = F \in m\mathcal{SS}_{\mathcal{D}_\Su}(G)$. Let $T$ be a spanning tree such that $T \geq F$, which has $C^\ast$ as a fundamental cocycle. Assume without loss of generality that $\sigma_\Su^\ast(\vec{C}^\ast)= +$. Because $\T$ refines $\Su$, we also have $\sigma_\T^\ast(\vec{C}^\ast) = +$, and $\vec{C}^\ast \subseteq \vec{T^\ast} = \chi^\ast(\tau)$ for some $\tau \in \T$. It is clear from Theorem \ref{thm: triang-to-atlas} that if $\vec{C}^\ast \subseteq \chi^\ast(\tau) = \vec{T^\ast}$ and $C^\ast$ is a fundamental cocycle of $T$, the rank of $\tau\cap F_{\vec{C}^\ast}$ is the rank of $F_{\vec{C}^\ast}$. Additionally, because $\T$ is a triangulation, $\tau \cap F_{\vec{C}^\ast}$ is a simplex and the restriction of $\T$ to $F_{\vec{C}^\ast}$ is nontrivial. The proof follows from $\tau \cap F_{\vec{C}^\ast} = c \cap F_{\vec{C}^\ast}$.
		
		For the forward direction, denote the restriction of $\Su$ to $F_{\vec{C}^\ast}$ by $\Su_{\vec{C}^\ast}$. By nontriviality, we have that $\sigma_\Su(\vec{C}^\ast) = \pm$ and $c \cap F_{\vec{C}^\ast}$ is a simplex. Assume without loss of generality that $\sigma^\ast_\Su(\vec{C}^\ast) = +$. By Lemma \ref{lem:lawrence_polytopes_lifting_subdivisions}, $c\cap F_{\vec{C}^\ast}$ is a cell of the triangulation $\T^+_{\vec{C}^\ast}$ of $\Lambda(\M^\ast(G))(F_{\vec{C}^\ast})$. Therefore, $c \cap F_{\vec{C}^\ast}$ is of the form $E(\Lambda(\M^\ast(G))(F_{\vec{C}^\ast})) \backslash e$ for some $e \in (\vec{C}^\ast)_+$. We then have $\vec{C}^\ast \subseteq \chi^\ast(c \cap F_{\vec{C}^\ast}) \subseteq \chi^\ast(c)$. Additionally, since $\chi^\ast(c\cap F_{\vec{C}^\ast})$ only has one one-way oriented edge, any spanning tree $T$ such that $T \geq_{\mathcal{SS}_{\mathcal{D}_\Su}(G)} \underline{\chi^\ast(c)}$ also only contains $e \in C^\ast$, and therefore $C^\ast$ is a fundamental cocycle of $T$.
	\end{proof}

	We can describe the cells of the subdivision more explicitly in terms of the refining triangulation. 
	
	\begin{lem} \label{lem::subdiv->atlases}
		Let $\Su$ be a subdivision of the Lawrence polytope $\Lambda(\M^\ast(G))$, and let $\chi^\ast$ be the map between cells of $\Su$ to internal orientations of $F \in m\mathcal{SS}_{\mathcal{D}_\Su}(G)$. Then, for each $F \in m\mathcal{SS}_{\mathcal{D}_\Su}(G)$, there exists a cell $c_F$ of $\Su$ such that $\chi^\ast(c_F) = \vec{F}^\ast$. 
	\end{lem}
	\begin{proof}
		Consider a forest $F \in m\mathcal{SS}_{\mathcal{D}_\Su}(G)$. Define
		\[T(F) \coloneqq \{T \in \mathcal{ST}(G):T \geq F \text{ in } \mathcal{SS}_{\mathcal{D}_\Su}(G)\}.\]
		Then, $E(F) = \cap_{T \in T(F)} E(T)$. Consider a triangulation $\T$ of $\Lambda(\M^\ast(G))$ that refines $\Su$, and consider a $T \in T(F)$. Then for edges in $E(T) \backslash E(F)$, the fundamental cocycles $C^\ast(T,e)$ satisfy $\sigma_\Su^\ast(\pm\vec{C}^\ast) =0$. Let $\tau \in \T$ be the maximal simplex of $\T$ such that $\chi^\ast(\tau) = \vec{T}^\ast$. We cannot have $\tau \in \Su$, because $\tau \cap F_{\vec{C}^\ast(T,e)}$ is a cell of the restricted triangulation of $F_{\vec{C}^\ast(T,e)}$ as described in Lemma \ref{lem:fundcocircuits-restricted triangs}. Therefore, there exists a cell $c$ of $\Su$ such that $c \cap F_{\vec{C}^\ast(T,e)} = F_{\vec{C}^\ast(T,e)}$, and $\tau \subset c$. From Lemma \ref{lem::subdiv->oriented forests}, $\chi^\ast(c) = \vec{F}_T^\ast$ for some forest $F_T \in m\mathcal{SS}_{\mathcal{D}_\Su}(G)$, but since $m\mathcal{SS}_{\mathcal{D}_\Su}(G)$ is ordered with respect to inclusion, we have $F_T = F$, as desired. We can then take $c_F = \cup_{T \in T(F)} (\chi^\ast)^{-1}(\vec{T}^\ast)$, where each $(\chi^\ast)^{-1}(\vec{T}^\ast)$ is a maximal simplex of $\T$. 
	\end{proof}
	
	\begin{thm}\label{thm::subdiv->atlases}
		Let $\Su$ be a subdivision of $\Lambda(\M^\ast(G))$. Then, the set $\{\chi^\ast(c):c \in \Su\}$ is a collection of internal orientations $\{\vec{F}^\ast:F \in m\mathcal{SS}_{\mathcal{D}_\Su}(G)\}$.
	\end{thm}
	\begin{proof}
		From Lemma \ref{lem::subdiv->oriented forests}, for every cell $c$ of $\Su$, $\chi^\ast(c)$ is an internally oriented forest $\vec{F}^\ast$ for $F \in m\mathcal{SS}_{\mathcal{D}_\Su}(G)$. Each $F \in m\mathcal{SS}_{\mathcal{D}_\Su}(G)$ receives an orientation by Lemma \ref{lem::subdiv->atlases}.
	\end{proof}

	\subsection{Constructing stability conditions from subdivisions of the Lawrence polytope}
	Theorem \ref{thm:triang-to-goodIG} shows that every triangulation of the Lawrence polytope $\Lambda(\M^\ast(G))$ corresponds to a nondegenerate stability condition. In this section, we generalize this result by showing that every subdivision $\Su$ of the Lawrence polytope $\Lambda(\M^\ast(G))$ yields a stability condition $\mathfrak{n}^{\Su}$. As we've done for triangulations in \cite{crepeau}, one could take the subdividing atlas in Theorem \ref{thm::subdiv->atlases}, construct a $\mathcal{D}_{\Su}$-forest function $I$, and count the number of divisors in $\mathcal{BD}_I(G)$. What makes this strategy difficult for subdivisions is that $c_{\mathcal{D}_\Su}(G)$ does not always have a nice combinatorial interpretation. Instead, for a subdivision $\Su$, we use a refining triangulation $\mathcal{T}$ to construct a stability condition, and then show the $\mathcal{D}$-forest function corresponding to the stability condition $\mathfrak{n}^\T$ is equivalent to the one constructed directly from the subdividing atlas.

	We first prove the following lemma. Throughout, let $\T$ be a triangulation of the Lawrence polytope $\Lambda(\M^\ast(G))$, which corresponds both to the orientable $\emptyset$-forest function $I^\T:\mathcal{ST}(G) \to \text{Div}^{-1}(G)$ and the stability condition $\mathfrak{n}^\T$, as constructed in Theorem \ref{thm:triang-to-goodIG}. The stability condition $\mathfrak{n}^\T$ is degree $g(G)-1$. Recall from Lemma \ref{lem:forest-fn-to-n} that for a biconnected subset $W$, if a spanning tree $T$ has $E(W,W^c)$ as a fundamental cocycle, $I^\T(T)_W = \mathfrak{n}^\T_W - g(G[W])$ for all such spanning trees. We then define $I^\T_W \coloneqq \mathfrak{n}^\T_W - g(G[W])$.
	
	\begin{lem} \label{lem::partition-I values}
		Let $\T$ be a triangulation of the Lawrence polytope $\Lambda(\M^\ast(G))$. Consider pairwise disjoint, biconnected sets $W_1$, $W_2$, $W_3$ such that $W_1 \cup W_2 \cup W_3 = V(G)$. Then, either $I_{W_i}^\T = -1$, $I_{W_j}^\T = -1$, and $I_{W_k}^\T = 0$ or $I_{W_i}^\T = -1$, $I_{W_j}^\T = 0$, and $I_{W_k}^\T = 0$ for $\{i,j,k\} = \{1,2,3\}$.
	\end{lem}
	\begin{proof}
		Because $\mathfrak{n}^{\mathcal{T}}$ is nondegenerate, by Definition \ref{def:stability conditon}, we have
		\begin{equation}\label{eqn:lemma-with-I}
			\sum_{i=1}^3 \mathfrak{n}^{\mathcal{T}}_{W_i} + \sum_{1 \leq i < j \leq 3} |E(W_i, W_j) |- (g(G)-1) \in \{1,2\}.
		\end{equation}
		
		We have $\mathfrak{n}_{W}^\T = I_{W}^\T + g(G[W])$ for every $W \in \text{BCon}(G)$. We observe that
		\begin{equation} \label{eqn:sum of genus}
			\sum_{i=1}^3 g(G[W_i]) = \sum_{i=1}^3 (|E(G[W_i])| - |W_i| + 1) = |E(G)| + 3  - |V(G)| - \sum_{1 \leq i < j \leq 3} E|(W_i, W_j)| 
		\end{equation}
		so Equation \ref{eqn:lemma-with-I} becomes $I^\T_{W_1} + I_{W_2}^\T + I_{W_3}^\T \in \{-1,-2\}$. The function $I^\T$ is orientable, so from Lemma \ref{lem:I-orientability}, each $I_{W_i}^\T \in \{0,-1\}$, so either $I_{W_i}^\T = -1$, $I_{W_j}^\T = -1$, and $I_{W_k}^\T = 0$, or $I_{W_i}^\T = -1$, $I_{W_j}^\T = 0$, and $I_{W_k}^\T = 0$. 
	\end{proof}
	
	We can also consider the single-element extension of $\M(G)$ corresponding to a triangulation $\T$ of $\Lambda(\M^\ast(G))$, and the localization $\sigma_\T^\ast$ of that extension. Let $W_i$, $W_j$, and $W_k$ be pairwise disjoint biconnected subsets such that $W_i \cup W_j \cup W_k = V(G)$. By Lemma \ref{lem:rank2minorexists}, there exists a rank $2$ minor of $\M(G)$, called $\mathcal{N}$, which is the graphic matroid of the graph minor in Figure \ref{fig:matroidminor}.
	
	\begin{figure}[h]
		\centering
		\begin{tikzcd}
			&& {W_i} && \\
			& {} \\
			{W_k} &&&& {W_j}
			\arrow["{E(W_i, W_k)}"{description}, dashed, no head, from=1-3, to=3-1]
			\arrow["{E(W_i, W_j)}"{description}, dashed, no head, from=1-3, to=3-5]
			\arrow["{E(Wj, W_k)}"{description}, dashed, no head, from=3-1, to=3-5]
		\end{tikzcd}
		\caption{Minor of $G$ with cocycles $E(W_i, W_i^c), E(W_j,W_j^c), E(W_k,W_k^c)$}
		\label{fig:matroidminor}
	\end{figure}
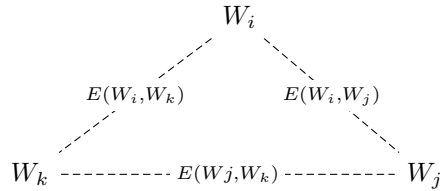
	
	We can then consider the extension defined by $\sigma^\ast_\T$ on $\mathcal{N}$. 
	\begin{lem} \label{lem:restiction_configs}
		Let $\sigma^\ast_\T$ be a localization of the extension of $\M(G)$ corresponding to the triangulation $\T$.  
		Let $\mathcal{N}$ be the minor of $\M(G)$ with cocircuits $E(W_i,W_i^c)$, $E(W_j, W_j^c)$, and $E(W_k, W_k^c)$, and let $\vec{C}^\ast_W$ be the cocircuit $E(W,W^c)$ oriented towards $W$. Then, $\sigma^\ast_\T$ yields one of two configurations in Figure \ref{fig:possible configs} on the minor $\mathcal{N}$. 
		
		\begin{figure}
			\centering
			\begin{subfigure}{.4\textwidth}
				\centering
				\begin{tikzpicture}
					\draw (0,2)--(0,-2);
					\draw (-2,1.75)--(2,-1.75);
					\draw (2,1.75)--(-2,-1.75);
					\node[anchor = north west] at (0,2) {\scriptsize $-\vec{C}_{W_k}^\ast$};
					\node[anchor = south] at (0,2) {$-$};
					\node[anchor = north] at (0,-2) {$+$};
					\node[anchor = south east] at (0,-2) {\scriptsize $\vec{C}_{W_k}^\ast$};
					
					\node[anchor = south west] at (2,1.75) {$-$};
					\node[anchor = north west] at (2,1.75) {\scriptsize $\vec{C}_{W_i}^\ast$};
					
					\node[anchor = north east] at (-2,-1.75) {$+$}; 
					\node[anchor = south east] at (-2,-1.75) {\scriptsize $-\vec{C}_{W_i}^\ast$};

					\node[anchor = south east] at (-2,1.75) {$-$};
					\node[anchor = west] at (-2,1.75) {\scriptsize $\vec{C}_{W_j}^\ast$};
					\node[anchor = north west] at (2,-1.75) {$+$};
					\node[anchor = east] at (2,-1.75) {\scriptsize $-\vec{C}_{W_j}^\ast$};
				\end{tikzpicture}
				\subcaption{Configuration A}
			\end{subfigure}
			\begin{subfigure}{.4\textwidth}
				\centering
				\begin{tikzpicture}
					\draw (0,2)--(0,-2);
					\draw (-2,1.75)--(2,-1.75);
					\draw (2,1.75)--(-2,-1.75);
					\node[anchor = north west] at (0,2) {\scriptsize $-\vec{C}_{W_k}^\ast$};
					\node[anchor = south] at (0,2) {$-$};
					\node[anchor = north] at (0,-2) {$+$};
					\node[anchor = south east] at (0,-2) {\scriptsize $\vec{C}_{W_k}^\ast$};
					
					\node[anchor = south west] at (2,1.75) {$-$};
					\node[anchor = north west] at (2,1.75) {\scriptsize $\vec{C}_{W_i}^\ast$};
					
					\node[anchor = north east] at (-2,-1.75) {$+$}; 
					\node[anchor = south east] at (-2,-1.75) {\scriptsize $-\vec{C}_{W_i}^\ast$};

					\node[anchor = south east] at (-2,1.75) {$+$};
					\node[anchor = west] at (-2,1.75) {\scriptsize $\vec{C}_{W_j}^\ast$};
					\node[anchor = north west] at (2,-1.75) {$-$};
					\node[anchor = east] at (2,-1.75) {\scriptsize $-\vec{C}_{W_j}^\ast$};
				\end{tikzpicture}
				\subcaption{Configuration B}
			\end{subfigure}
			\caption{Possible configurations}
			\label{fig:possible configs}
		\end{figure}
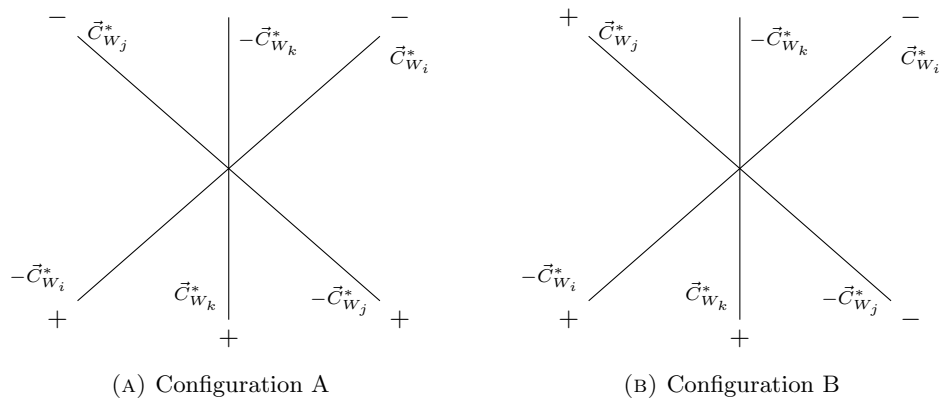
	\end{lem}
	\begin{proof}
		We consider the two possibilities for $I_{W_i}^\T, I_{W_j}^\T, I_{W_k}^\T$ outlined in Lemma \ref{lem::partition-I values}. It follows from the definition of localizations that if $I_{W}^\T = -1$, then $\sigma_\T^\ast(-\vec{C}^\ast_{W}) = +$, and if $I_{W}^\T = 0$, then $\sigma_\T^\ast(\vec{C}^\ast_{W}) = +$. 
		
		First, consider if $I_{W_i}^\T = I_{W_j}^\T = -1$ and $I_{W_k}^\T = 0$. Then, $\sigma_\T^\ast(-\vec{C}_{W_i}^\ast) = \sigma_\T^\ast(-\vec{C}_{W_j}^\ast) = \sigma_\T^\ast(\vec{C}_{W_k}^\ast)= +$. This corresponds to Configuration A in Figure \ref{fig:possible configs}; for example, the orientation of the edges $E(W_i, W_j)$ in $-\vec{C}_{W_i}^\ast$ is opposite of the orientation of those edges in $-\vec{C}_{W_j}^\ast$, so $-\vec{C}_{W_i}^\ast$ and $-\vec{C}_{W_j}^\ast$ should be on the opposite sides of the hyperplane corresponding to the signed cocircuits $\pm\vec{C}_{W_k}^\ast$. If instead $I_{W_i}^\T = -1$ and $I_{W_j}^\T = I_{W_k}^\T = 0$, we have Configuration B. 
	\end{proof}

	\begin{thm}\label{thm:stability-condition-through-refinement}
		Let $\Su$ be a subdivision of the Lawrence polytope $\Lambda(\M^\ast(G))$. Let $\mathcal{T}$ be a triangulation of $\Lambda(\M^\ast(G))$ that refines $\Su$, and consider the associated stability condition $\mathfrak{n}^\T$. Let $\sigma_\T^\ast$ be the localization of the extension of $\M(G)$ corresponding to $\T$, and let $\sigma^\ast_\Su$ be the localization of the extension of $\M(G)$ corresponding to $\Su$. 
		
		Define $\mathfrak{n}^\Su: \text{BCon}(G) \to \mathbb{Z}$ by
		\[\mathfrak{n}^\Su_W = \begin{cases} \mathfrak{n}^\T_W - 1 & W \in \mathcal{D}_\Su \text{ and } \sigma_\T^\ast(\vec{C}^\ast_W) = +, \\
			\mathfrak{n}^\T_W & \text{otherwise.}
		\end{cases}\]
		
		Then, $\mathfrak{n}^\Su$ is a degenerate stability condition of degree $g(G) - 1$ such that $\mathcal{D}(\mathfrak{n}^\Su) = \mathcal{D}_\Su$. 
	\end{thm}
	\begin{proof}
		First, we check that 
		\[\mathfrak{n}^\Su_W + \mathfrak{n}^\Su_{W^c} + |E(W, W^c)| - (g(G) - 1) \in \{0,1\}.\]
		Since $\mathfrak{n}^{\mathcal{T}}$ is a nondegenerate stability condition, we have
		\[\mathfrak{n}_W^{\mathcal{T}} + \mathfrak{n}_{W^c}^{\mathcal{T}} + |E(W,W^c)| - g(G) + 1 =1\]
		for every $W \in \text{BCon}(G)$. If $W \notin \mathcal{D}_\Su$, then by definition $W^c \notin \mathcal{D}_\Su$, so we have $\mathfrak{n}_W^\Su =\mathfrak{n}^{\mathcal{T}}_W$ and $\mathfrak{n}^\Su_{W^c} =\mathfrak{n}^{\mathcal{T}}_{W^c}$ and the desired equality holds. If $W \in \mathcal{D}_\Su$, then either $\mathfrak{n}^\Su_W =\mathfrak{n}^{\mathcal{T}}_W -1$ or $\mathfrak{n}^\Su_{W^c} =\mathfrak{n}^{\mathcal{T}}_{W^c}-1$, and the desired equality holds. By construction, $\mathcal{D}(\mathfrak{n}^\Su) = \mathcal{D}_\Su$. 
		
		Now, consider biconnected, pairwise disjoint subsets $W_1, W_2, W_3$ such that $W_1 \cup W_2 \cup W_3 = V(G)$. If $W_i$ and $W_j$ are $\mathfrak{n}^\Su$-degenerate, both $W_i$ and $W_j$ are in $\mathcal{D}_\Su$. It follows from the definition of a degeneracy set that $W_k = (W_i \cup W_j)^c \in \mathcal{D}_\Su$, and therefore $W_k$ is $\mathfrak{n}^\Su$-degenerate. Finally, we show that
		\begin{equation}\label{eqn:subdiv-thm-tribond}
			\sum_{i=1}^3 \mathfrak{n}^\Su_{W_i} + \sum_{1 \leq i < j \leq 3} |E(W_i, W_j)| - (g(G) -1) = \begin{cases}
				\{1,2\} & W_i \text{ is } \mathfrak{n}^\Su\text{-nondegenerate for all } i \\
				\{1\} & \exists! i \text{ such that } W_i \text{ is } \mathfrak{n}^\Su\text{-degenerate} \\
				\{0\} & W_i \text{ is } \mathfrak{n}^\Su\text{-degenerate for all } i.
			\end{cases}
		\end{equation}
		
		In the first case, we have $\mathfrak{n}^\Su_{W_i} = \mathfrak{n}_{W_i}^{\mathcal{T}}$ for $i = 1,2,3$, and the condition is satisfied because $\mathfrak{n}^\T$ is a stability condition. In the last case, we have $\mathfrak{n}^\Su_{W_i} = -1  + g(G[W_i])$ for all $i$, so using Equation \ref{eqn:sum of genus}, we see Equation \ref{eqn:subdiv-thm-tribond} is $0$. Finally, we consider when there exists a unique $i$ such that $W_i$ is $\mathfrak{n}^\Su$-degenerate. Without loss of generality, assume $W_1$ is the unique $\mathfrak{n}^\Su$-degenerate set. Let $I^\T$ be the function on spanning trees constructed from the triangulation $\T$. Then, Lemma \ref{lem::partition-I values} gives us two possibilities for $I_{W_i}^\T, I_{W_j}^\T, I_{W_k}^\T$, where $\{i,j,k\} = \{1,2,3\}$:
		\begin{enumerate}
			\item $I_{W_i}^\T = -1, I_{W_j}^\T=-1$, $I_{W_k}^\T = 0$;
			\item $I_{W_i}^\T = -1, I_{W_j}^\T = 0, I_{W_k}^\T = 0$.
		\end{enumerate}
		
		In the first case, we have
		\[\sum_{i=1}^3 \mathfrak{n}^\T_{W_i} + \sum_{1 \leq i < j \leq 3} |E(W_i, W_j)| - g(G) +1 = 1.\]
		Therefore, we wish to show that $\mathfrak{n}^\Su_{W_i} = \mathfrak{n}_{W_i}^\T$ for each $i \in \{1,2,3\}$. By Lemma \ref{lem:restiction_configs}, the localization $\sigma_\T^\ast$ yields Configuration A on the minor of $\M(G)$ corresponding to contracting all edges $E(G) \backslash (E(W_i, W_j) \cup E(W_i,W_k) \cup E(W_j,W_k))$. If $W_k$ is $\mathfrak{n}^\Su$-degenerate, then Configuration A becomes the second forbidden subconfiguration in Figure \ref{fig:forbiddenconfigs}; therefore, either $W_i$ or $W_j$ is $\mathfrak{n}^\Su$-degenerate. In either case, we have $\sigma_\T^\ast(\vec{C}_{W_i}^\ast) = \sigma_\T^\ast(\vec{C}_{W_j}^\ast) = -$, so $\mathfrak{n}^\Su_{W_i} = \mathfrak{n}_{W_i}^\T$ for each $i \in \{1,2,3\}$.
		For the second case, we have
		\[\sum_{i=1}^3 \mathfrak{n}^\T_{W_i} + \sum_{1 \leq i < j \leq 3} |E(W_i, W_j)| - g(G) +1 = 2.\] 
		By Lemma \ref{lem:restiction_configs}, the localization $\sigma_\T^\ast$ yields Configuration B on the minor of $\M(G)$ corresponding to contracting all edges $E(G) \backslash (E(W_i, W_j) \cup E(W_i,W_k) \cup E(W_j,W_k))$. If $W_i$ is $\mathfrak{n}^\Su$-degenerate, then Configuration B becomes the second forbidden subconfiguration in Figure \ref{fig:forbiddenconfigs}; therefore, exactly one of $W_j$ or $W_k$ is $\mathfrak{n}^\Su$-degenerate. Without loss of generality, assume $W_j$ is $\mathfrak{n}^\Su$-degenerate. We have $\sigma_\T^\ast(\vec{C}_{W_j}^\ast) = +$, so $\mathfrak{n}^\Su_{W_j} = \mathfrak{n}^\T_{W_j} - 1$ and the desired equality holds. Therefore, $\mathfrak{n}^\Su$ is a stability condition with degeneracy set $\mathcal{D}_\Su$. 
	\end{proof}

	Recall from Lemma \ref{lem:forest-fn-from-n} that for any stability condition $\mathfrak{n}$ with degeneracy set $\mathcal{D}(\mathfrak{n})$, we can construct a $\mathcal{D}(\mathfrak{n})$-forest function $I_{\mathfrak{n}}$ such that $|\mathcal{BD}_{I_{\mathfrak{n}}}(G)| = c_{\mathcal{D}(\mathfrak{n})}(G)$. For the stability condition $\mathfrak{n}^\Su$, we can explicitly describe $I_{\mathfrak{n}^\Su}$.
	
	\begin{lem}\label{lem:I-from-n}
		Let $\Su$ be a subdivision of $\Lambda(\M^\ast(G))$, and consider the stability condition $\mathfrak{n}^\Su$, as defined in Theorem \ref{thm:stability-condition-through-refinement}, with degeneracy set $\mathcal{D}_\Su$. Consider a forest $F \in m\mathcal{SS}_{\mathcal{D}_\Su}(G)$. For $e \in F$, recall that $F \backslash e = F[W_e] \sqcup F[V(F_{i_e})\backslash W_e] \sqcup_{i \neq i_e}F_i$, where $F_{i_e}$ is the connected component of $F$ containing $e$. Then, $I_{\mathfrak{n}^\Su}(F)$ is the unique divisor satisfying 
		\[\begin{cases}
			I_{\mathfrak{n}^\Su}(F)_{W_e} \in \{0, -1\} \\
			I_{\mathfrak{n}^\Su}(F)_{V(F_{i_e})\backslash W_e} \in \{0, -1\} \backslash I_{\mathfrak{n}^\Su}(F)_{W_e} \\
			I_{\mathfrak{n}^\Su}(F)_{V(F_i)} = -1.
		\end{cases}\]
	\end{lem}
	\begin{proof}
		Recall from Lemma \ref{lem:forest-fn-from-n} that $I_{\mathfrak{n}^\Su}(F)$ is the unique divisor on $G$ satisfying
		\[\begin{cases}
			I_{\mathfrak{n}^\Su}(F)_{W_e} = \mathfrak{n}^\Su(F_{i_e})_{W_e} \\
			I_{\mathfrak{n}^S}(F)_{V(F_{i_e})\backslash W_e} = \mathfrak{n}^\Su(F_{i_e})_{V(F_{i_e}) \backslash W_e} \\
			I_{\mathfrak{n}^\Su}(F)_{V(F_i)} = |\mathfrak{n}^\Su(F_i)|.
		\end{cases}\]
		
		We first consider $I_{\mathfrak{n}^\Su}(F)_{V(F_i)} = |\mathfrak{n}^\Su(F_i)|$. Because $F \in m\mathcal{SS}_{\mathcal{D}_\Su}(G)$, each $F_i$ is a connected, $\mathcal{D}_\Su$-admissible subgraph of $G$. Let $E_i \subset E(G)$ be the edges such that $F_i = G[V(F_i)] \backslash E_i$. We note that $|E_i| = g(G[V(F_i)])$, because $E_i$ is the complement of the spanning tree $F_i$ of $G[V(F_i)]$. From Lemma~\ref{lem:restriction-n-to-subgraphs}, we have $|\mathfrak{n}^\Su(F_i)| = \mathfrak{n}^\Su_{V(F_i)} - |E_i|$. We then wish to determine $\mathfrak{n}^\Su_{V(F_i)}$.
		
		Since $F_i$ is connected, $V(F_i)$ is a connected subset of $V(G)$. We can extend $\mathfrak{n}^\Su$ to be a function on connected subsets of $V(G)$, as defined in Definition~\ref{def:extended-V-fn}. Let $G[V(F_i)^c] = \bigsqcup_{j \in J} G[Z_j]$ be the decomposition of $G[V(F_i)^c]$ into biconnected subsets $Z_j$. We define
		\begin{equation}\label{equation:extension-to-connected}
			n^\Su_{V(F_i)} = g(G) - 1 - \sum_{j \in J}\mathfrak{n}^\Su_{Z_j} - |E(V(F_i), V(F_i)^c)| + |\{Z_j | Z_j \notin \mathcal{D}_\Su\}|.
		\end{equation}
		Because $F_i$ is a connected component of $F$, and $F$ is $\mathcal{D}_\Su$-admissible, $V(F_i) \in \hat{\mathcal{D}}_\Su$, and every $Z_j$ in the decomposition of $G[V(F_i)^c]$ is in $\mathcal{D}_\Su$. Therefore, $|\{Z_j | Z_j \notin \mathcal{D}_\Su\}| = 0$. By the definition of $\mathfrak{n}^\Su$ in Theorem \ref{thm:stability-condition-through-refinement}, for each $Z_j$ in the decomposition, we have $\mathfrak{n}^\Su_{Z_j} = -1 + g(G[Z_j]) = |E(G[Z_j])| - |Z_j|$. Using the decomposition $G[V(F_i)^c] = \bigsqcup_{j \in J} G[Z_j]$, we have
		\begin{align*}
			\sum_{j \in J} |Z_j| &= |V(F_i)^c| = |V(G)| - |V(F_i)| \\
			\sum_{j \in J} |E(G[Z_j])| &= |E(G)| - |E(V(F_i), V(F_i)^c)| - |E(G[V(F_i)])|. 
		\end{align*}
		After plugging into Equation~\ref{equation:extension-to-connected} and simplifying, we have $n^\Su_{V(F_i)} = |E(G[V(F_i)])| - |V(F_i)|$, so $|\mathfrak{n}^\Su(F_i)| = -1$.
		
		Now, we consider $I_{\mathfrak{n}^\Su}(F)_{W_e} = \mathfrak{n}^\Su(F_{i_e})_{W_e}$. Again, $W_e$ is a connected subset of $G$, so $G[W_e^c] = \sqcup_{j \in J} G[Z_j]$ for biconnected subsets $Z_j$ of $V(G)$. We first examine $\mathfrak{n}^\Su_{W_e}$. As before,
		\begin{equation} \label{equation-b}
			\mathfrak{n}^\Su_{W_e} = g(G) - 1 - \sum_{j \in J}\mathfrak{n}^\Su_{Z_j} - |E(W_e,W_e^c)| + |\{Z_j | Z_j \notin \mathcal{D}_\Su\}|.
		\end{equation}
		
		$W_e$ is not the vertex set of a connected component of $F$, since it is contained in $V(F_{i_e})$. Since $F$ is minimal in $\mathcal{SS}_{\mathcal{D}_\Su}(G)$, we also have $W_e \notin \hat{\mathcal{D}}_\Su$. Therefore, for the decomposition of $G[W_e^c]$ into disjoint $G[Z_j]$ where all $Z_j$ are biconnected, there exists a $Z_e$ such that $Z_e \notin \mathcal{D}_\Su$. Additionally, because all of the $G[Z_j]$ are disjoint, $Z_e$ is the unique biconnected subset in the decomposition of $G[W_e^c]$ such that $e \in E(Z_e, Z_e^c)$. Consider the localization $\sigma^\ast_\Su$ of the extension of $\M(G)$ corresponding to $\Su$. From Theorem \ref{thm:stability-condition-through-refinement}, we have
		\[ \mathfrak{n}^\Su_{Z_j} = \begin{cases}
			g(G[Z_j]) & \text{if } Z_j \notin \mathcal{D}_\Su \text{ and } \sigma_\Su^\ast(\vec{C}^\ast_{Z_j})=+\\
			-1 + g(G[Z_j]) & \text{otherwise.}
		\end{cases}\]
		This allows us to rewrite Equation~\ref{equation-b} as
		\begin{align*}
			\sum_{j \in J} \mathfrak{n}^\Su_{Z_j} &= \sum_{j \in J} g(G[Z_j]) - |\{ Z_j | Z_j \in \mathcal{D}_\Su\}| - |\{Z_j |  Z_j \notin \mathcal{D}_\Su \text{ and } \sigma_\Su^\ast(\vec{C}_{Z_j}^\ast)=-\}| 
			\\
			&= |J| + |E(G)| - |E(W_e, W_e^c)| - |E(G[W_e])| - |V(G)| + |W_e| -|\{ Z_j | Z_j \in \mathcal{D}_\Su\}| \\
			& \hspace{3cm}-|\{Z_j |  Z_j \notin \mathcal{D}_S \text{ and } \sigma_\Su^\ast(\vec{C}_{Z_j}^\ast)=- \} |\\
			&= |\{Z_j |  Z_j \notin \mathcal{D}_S \text{ and } \sigma_\Su^\ast(\vec{C}_{Z_j}^\ast)= + \}| + g(G) - 1 - |E(W_e, W_e^c)| - |E(G[W_e])| + |W_e|.
		\end{align*}
		
		Plugging back into our expression for $\mathfrak{n}^\Su_{W_e}$, we have
		\begin{align*}
			\mathfrak{n}^S_{W_e} &= |\{Z_j |  Z_j \notin \mathcal{D}_\Su \text{ and } \sigma_\Su^\ast(\vec{C}_{Z_j}^\ast)= - \}| + |E(G[W_e])| - |W_e| \\
			&= |\{Z_j |  Z_j \notin \mathcal{D}_S \text{ and } \sigma_\Su^\ast(\vec{C}_{Z_j}^\ast)= -\}| + g(G[W_e]) - 1.
		\end{align*}
		Let $E_{i_e} \subset E(G)$ such that $F_{i_e} = G[V(F_{i_e})] \backslash E_{i_e}$. Then, $|E_{i_e} \cap E(G[W_e])| = g(G[W_e])$. We then have 
		\begin{align*}
			\mathfrak{n}(F_{i_e})_{W_e} &= \mathfrak{n}_{W_e} - |E_{i_e} \cap E(G[W_e])| \\
			&= |\{Z_j |  Z_j \notin \mathcal{D}_S \text{ and } \sigma_\Su^\ast(\vec{C}_{Z_j}^\ast)= - \}| -1.
		\end{align*}
		We know $Z_e \notin \mathcal{D}_\Su$. Then, for all other $Z_j \ne Z_e$, we have that $E(Z_j, Z_j^c) \subseteq E(G) \backslash E(F_{i_e})$. Assume that $G[Z_j] \subseteq G[V(F_{i_e}^c)]$. Since $V(F_{i_e})$ is $\mathcal{D}_\Su$-admissible, we have a decomposition $G[V(F_{i_e})^c] = \bigsqcup_{k \in K} G[W_k]$ where each $W_k \in \mathcal{D}_\Su$. If there exists a $W_k$ such that $Z_j = W_k$, then $Z_j$ is $\mathfrak{n}^\Su$-degenerate, as desired. Because $Z_j$ is biconnected, the other possibility is there exists $W_k$ such that $Z_j \subset W_k$. However, in this case $G[W_k]$ would also be in the decomposition of $G[W_e^c]$, so such a $Z_j$ cannot exist.
		
		Therefore, $|\{Z_j |  Z_j \notin \mathcal{D}_\Su \text{ and } \sigma_\Su^\ast(\vec{C}_{Z_j}^\ast)= - \}|$ is either $0$ or $1$, depending on the orientation of $C_{Z_e}^\ast$, so $I_{\mathfrak{n}^\Su}(F)_{W_e} \in \{0, -1\}$.  By the same calculation, we have that $I_{\mathfrak{n}^\Su}(F)_{V(F_{i_e})\backslash W_e} \in \{0, -1\} \backslash I_{\mathfrak{n}^\Su}(F)_{W_e}$, as desired.
	\end{proof}
	
	\begin{corollary} \label{cor:I-from-n}
		Let $\Su$ be a subdivision of $\Lambda(\M^\ast(G))$, and $\mathfrak{n}^\Su$ be the associated stability condition with degeneracy set $\mathcal{D}_\Su$. Then, the $\mathcal{D}_\Su$-forest function $I_{\mathfrak{n}^\Su}$ is orientable. 
	\end{corollary}
	
	\begin{proof}
		This follows immediately from Lemmas \ref{lem:I-orientability} and \ref{lem:I-from-n}.
	\end{proof}
	
	Lemma \ref{lem:I-from-n} and Corollary \ref{cor:I-from-n} allow us to generalize Theorem \ref{thm:triang-to-goodIG} to subdivisions of the Lawrence polytope $\Lambda(\M^\ast(G))$. We show this agrees from constructing a $\mathcal{D}_\Su$-forest function directly from the subdivision. 
	\begin{thm}
		Let $\Su$ be a subdivision of the Lawrence polytope $\Lambda(\M^\ast(G))$. Consider the map $\chi^\ast$ from cells of the subdivision $\Su$ to internal orientations $\vec{F^{\ast}}$ for each $F \in m\mathcal{SS}_{\mathcal{D}_\Su}(G)$. Then, define the $\mathcal{D}_\Su$-forest function $I^\Su$ by
		\begin{align*}
			I^\Su:m\mathcal{SS}_{\mathcal{D}_\Su}(G) &\to \Div^{-b_0(F)}(G) \\ 
			F &\mapsto D_{-(\vec{F^\ast})^c}
		\end{align*}
		and $\mathcal{BD}_{I^\Su}(G)$ satisfies $|\mathcal{BD}_{I^\Su}(G)| = c_{\mathcal{D}_\Su}(G)$. 
	\end{thm}
	
	\begin{proof}
		From Theorem \ref{thm::subdiv->atlases}, the subdivision $\Su$ yields an internal orientation $\vec{F^\ast}$ of every $F \in m\mathcal{SS}_{\mathcal{D}_\Su}(G)$. The one-way oriented edges in $\vec{F}^\ast$ are described by the localization $\sigma^\ast_\Su$, and therefore $I^\Su$ is the $\mathcal{D}_\Su$-forest function $I_{\mathfrak{n}^\Su}$ above.
	\end{proof}
	
	\begin{example}\label{ex:ndeg}
		Let $\Su_{\text{triv}}$ be the trivial subdivision of $\Lambda(\M^\ast(G))$, where the only cell is the Lawrence polytope itself. Then $\chi^\ast(\Su_{\text{triv}})$ is the fourientation $\vec{G^\ast}$ where every edge of $G$ is bioriented. The localization $\sigma^\ast_\Su$ sends every signed cocircuit to $0$. Therefore, the degeneracy subset of $\Su_\text{triv}$ is $\text{BCon}(G)$, the set of all biconnected subsets of $G$. 
		We then have $I^\Su(G) = D_{-(\vec{G^\ast})^c} = \begin{bmatrix} -1 & -1 & \cdots & -1 \end{bmatrix}^T \in \mathbb{Z}^{|V(G)|}$. The set $\mathcal{BD}_{I^\Su}(G)$ is all orientable divisors $D_\Or$ on $G$. By Lemma 40 of \cite{bernardi-bij}, $|\mathcal{BD}_{I^\Su}(G)| = |\text{spanning forests of } G|$. The function $I^\Su$ corresponds to the maximally degenerate stability condition $\mathfrak{n}^{\deg}$ defined in Lemma 5.2 of \cite{fava-pagani-viviani2024}. 
	\end{example}
	
	Given an orientable $\mathcal{D}$-forest function $I$ satisfying $|\mathcal{BD}_I(G)| = c_{\mathcal{D}}(G)$, we can construct a subdivision $\Su$ of $\Lambda(\M^\ast(G))$ such that $I = I^\Su$. 
	\begin{thm}\label{thm:subdiv-to-forest}
		Let $I$ be an orientable $\mathcal{D}$-forest function such that $|\mathcal{BD}_I(G)| = c_{\mathcal{D}}(G)$. Then $\mathcal{D}$ is an orientable degeneracy set, and $I = I^\Su$ for some subdivision $\Su$ of $\Lambda(\M^\ast(G))$.
	\end{thm}
	\begin{proof}
		Assume $I$ is an orientable $\mathcal{D}$-forest function such that $|\mathcal{BD}_I(G)| = c_{\mathcal{D}}(G)$. We recall from Corollaries \ref{cor:stab-degree-from-orientable-I} and \ref{cor:stab-values-from-orientable-I} that $\mathfrak{n}^I$ is a degree $g(G)-1$ stability condition defined by		
		\begin{equation*}
			\mathfrak{n}^I_W = \begin{cases}
				g(G[W]) - 1 & W \in \mathcal{D} \\
				g(G[W]) & W \notin \mathcal{D} \text{ and } I(F)_W = 0 \text{ for all forests satisfying Equation \ref{eqn:adapted-forest}} \\
				g(G[W])-1 & W \notin \mathcal{D} \text{ and } I(F)_W = -1 \text{ for all forests satisfying Equation \ref{eqn:adapted-forest}}. 	
			\end{cases}
		\end{equation*}
		Let $\vec{C}^\ast_W$ be the signed cocircuit $E(W, W^c)$ oriented towards $W$. We define $\sigma^\ast: \mathcal{C}^\ast(\M) \to \{+,-,0\}$ in the following way:
		\[\sigma^\ast(\vec{C}^\ast_W) = \begin{cases}
			0 & W \in \mathcal{D} \\
			+ & W \notin \mathcal{D} \text{ and } \mathfrak{n}^I_W = g(G[W]) \\
			- & \text{otherwise.}
		\end{cases}\]
		We then show that $\sigma^\ast$ is a localization by showing that it cannot produce one of the three forbidden subconfigurations in Figure \ref{fig:forbiddenconfigs}. Let $W_i$, $W_j$, and $W_k$ be pairwise disjoint biconnected subsets such that $W_i \cup W_j \cup W_k = V(G)$, and consider the rank two minor of $\M(G)$ with vertices $W_i$, $W_j$, and $W_k$, as pictured in Figure \ref{fig:matroidminor}. Because $\mathcal{D}$ is a degeneracy set, we cannot have $\sigma^\ast(\pm \vec{C}^\ast_{W_i}) = \sigma^\ast(\pm\vec{C}^\ast_{W_j}) = 0$ and $\sigma^\ast(\vec{C}_{W_k}^\ast) = \pm$, so $\sigma^\ast$ does not create the third forbidden subconfiguration in Figure \ref{fig:forbiddenconfigs}. 
		
		If $\sigma^\ast$ produces the first forbidden subconfiguration in Figure \ref{fig:forbiddenconfigs}, we either have $\sigma^\ast(\vec{C}_{W_i}^\ast) = \sigma^\ast(\vec{C}_{W_j}^\ast) = \sigma^\ast(\vec{C}_{W_k}^\ast) = +$ or $\sigma^\ast(\vec{C}_{W_i}^\ast) = \sigma^\ast(\vec{C}_{W_j}^\ast) = \sigma^\ast(\vec{C}_{W_k}^\ast) = -$. In either case, $W_i$, $W_j$, and $W_k$ are all not $\mathfrak{n}^I$-degenerate, so the stability condition $\mathfrak{n}^I$ must satisfy
		$$ \mathfrak{n}^I_{W_i} + \mathfrak{n}^I_{W_j} + \mathfrak{n}^I_{W_k} + |E(W_i, W_j)| + |E(W_i, W_k)| + |E(W_j, W_k)|- g(G) +1 \in \{1,2\}.$$
		However in the first case, $\mathfrak{n}^I_{W_i} = g(G[W_i])$, $\mathfrak{n}^I_{W_j} = g(G[W_j])$, and $\mathfrak{n}^I_{W_k} = g(G[W_k])$, so
		$$ \mathfrak{n}^I_{W_i} + \mathfrak{n}^I_{W_j} + \mathfrak{n}^I_{W_k} + |E(W_i, W_j)| + |E(W_i, W_k)| + |E(W_j, W_k)|- g(G) +1 = 3.$$
		In the second case, $\mathfrak{n}^I_{W_i} = g(G[W_i])-1$, $\mathfrak{n}^I_{W_j} = g(G[W_j])-1$, and $\mathfrak{n}^I_{W_k} = g(G[W_k])-1$, so
		$$ \mathfrak{n}^I_{W_i} + \mathfrak{n}^I_{W_j} + \mathfrak{n}^I_{W_k} + |E(W_i, W_j)| + |E(W_i, W_k)| + |E(W_j, W_k)|- g(G) +1 = 0.$$
		Therefore, $\sigma^\ast$ cannot produce the first forbidden subconfiguration. Finally, for the remaining forbidden subconfiguration, assume without loss of generality that $\sigma^\ast(\pm\vec{C}^\ast_{W_i}) = 0$, $\sigma^\ast(\vec{C}_{W_j}^\ast) = -$, and $\sigma^\ast(\vec{C}_{W_k}^\ast) = -$. Since $W_i$ is uniquely $\mathfrak{n}^I$-degenerate, the stability condition $\mathfrak{n}^I$ must satisfy
		$$ \mathfrak{n}^I_{W_i} + \mathfrak{n}^I_{W_j} + \mathfrak{n}^I_{W_k} + |E(W_i, W_j)| + |E(W_i, W_k)| + |E(W_j, W_k)|- g(G) +1 = 1$$.
		However, this implies $\mathfrak{n}^I_{W_i} = g(G[W_i])$, which contradicts the definition of $\mathfrak{n}^I$. Therefore, $\sigma^\ast$ is a localization, and it corresponds to a subdivision $\Su$ of $\Lambda(\M^\ast(G))$. By definition, $\mathcal{D}_\Su = \mathcal{D}$, and $\mathcal{D}$ is orientable. It follows from Theorem \ref{thm:stability-condition-through-refinement} that $\mathfrak{n}^\Su = \mathfrak{n}^I$, so $I^\Su = I$. 
		
	\end{proof}
	We can also generalize Theorem \ref{thm:acyclic-classical} to regular subdivisions of the Lawrence polytope $\Lambda(\M^\ast(G))$. 
	
	\begin{corollary}\label{cor:reg->clasical}
		Let $\Su$ be a regular subdivision of the Lawrence polytope $\Lambda(\M^\ast(G))$. Then, the stability condition $\mathfrak{n}^\Su$ is classical. 
	\end{corollary}
	
	\begin{proof}
		Recall that any regular subdivision of the Lawrence polytope $\Lambda(\M^\ast(G))$ can be refined to a regular triangulation $\T$ of $\Lambda(\M^\ast(G))$. We induct on the distance in the poset $\text{Subdiv}(\mathcal{P}^\ast)$ between the regular subdivision $\Su$ and a refining regular triangulation $\T$. First, consider if $\Su = \T$. From Theorem \ref{thm:acyclic-classical}, the triangulation $\T$ corresponds to a classical stability condition $\mathfrak{n}^\T = \mathfrak{n}(\phi^\T)$, where $\phi^\T$ is a generic numerical polarization on $G$. From Theorem \ref{thm:acyclic-classical}, $\phi^\T$ lives in a chamber in $V^{g-1}(G)$ defined by the hyperplanes
		\[\left\{\phi_W^\T - \frac{|E(W,W^c)|}{2} = \mathfrak{n}_W^\T \right\}_{W \in \text{BCon}(G)} \bigcup \left\{\phi_W^\T - \frac{|E(W,W^c)|}{2} = \mathfrak{n}_W^\T -1\right\}_{W \in \text{BCon}(G)}.\]
		
		Then, consider a regular subdivision $\Su'$ that is refined by the regular subdivision $\Su$, and such that there exists no subdivision $\Su''$ such that $\Su' \geq \Su'' \geq \Su$. We can then consider a polarization $\phi^\Su$ that lives in between the hyperplanes 
		\[\left\{\phi_W^\T - \frac{|E(W,W^c)|}{2} = \mathfrak{n}_W^\T = \mathfrak{n}^\Su_W\right\} \bigcup \left\{\phi_W^\T - \frac{|E(W,W^c)|}{2} = \mathfrak{n}_W^\T -1 = \mathfrak{n}_W^\Su -1\right\}\]
		for $W \notin \mathcal{D}_\Su$, and on the hyperplanes
		\[\left \{ \phi_W^\T - \frac{|E(W,W^c)|}{2} = \mathfrak{n}_W^\T - 1 = \mathfrak{n}^\Su_W\right \} = \left \{ \phi_{W^c}^\T - \frac{|E(W,W^c)|}{2} = \mathfrak{n}_{W^c}^\T = \mathfrak{n}^\Su_{W^c}\right \}\] for $W, W^c \in \mathcal{D}_\Su$, where $W$ is chosen such that $\sigma_\Su^\ast(\vec{C}^\ast_W) = +$. 
		
		We then consider $W \in \mathcal{D}_{\Su'} \backslash \mathcal{D}_\Su$. One can check that if there exists $W_1, W_2 \in \mathcal{D}_{\Su'}$ such that $W_1 \cap W_2=\emptyset$ and $W_1 \cup W_2$ is also biconnected, then \[\phi_{W_1\cup W_2}^{\Su'} - \frac{|E(W_1 \cup W_2, (W_1 \cup W_2)^c)|}{2} = g(G[W_1\cup W_2]) -1.\] 
		Therefore, we consider $W \in \mathcal{D}_{\Su'} \backslash \mathcal{D}_\Su$ that cannot be written as the union of disjoint biconnected sets $W_1$ and $W_2$ in $\mathcal{D}_{\Su'}$. If $W$ corresponds to a facet of the chamber that $\phi^\Su$ lives in, we move to that wall. Then, assume that the hyperplane $\{\phi_W^\Su - \frac{|E(W,W^c)|}{2} = g(G[W]) - 1\}$ is not a facet of the chamber $\phi^\Su$ lives in. Then, the value of $\phi_W^\Su - \frac{|E(W,W^c)|}{2}$ is determined by the facets of the chamber $\phi^\Su$ lives in, and $\sigma^\ast_\Su(\vec{C}^\ast_W)$ is fixed by some collection of $\sigma^\ast_{\Su}(\vec{C}_{W_{F_1}}^\ast)$ and $\sigma^\ast_{\Su}(\vec{C}_{W_{F_2}}^\ast)$, where the hyperplanes corresponding to $W_{F_1}$ and $W_{F_2}$ are facets of the chamber, $W_{F_1} \cap W_{F_2} = \emptyset$, and $W = (W_{F_1} \cup W_{F_2})^c$. Then, if $\sigma_{\Su'}^\ast(\pm \vec{C}_W^\ast) = 0$, we must also have $\sigma_{\Su'}^\ast(\pm\vec{C}_{W_F}^\ast) = 0$ for the biconnected subsets $W_F$, and $W_F \in \mathcal{D}_{\Su'}$, or else $\sigma_{\Su'}$ would yield the second forbidden subconfiguration in Figure \ref{fig:forbiddenconfigs}. Therefore, we can obtain $\phi^{\Su'}$ by moving to a wall of the chamber $\phi^\Su$ lives in. It follows that $\mathfrak{n}(\phi^{\Su'}) = \mathfrak{n}^{\Su'}$, and $\mathfrak{n}^{\Su'}$ is classical.
	\end{proof}
	\subsection{Non-orientable degeneracy sets and stability conditions}
	It is currently unknown if every degeneracy set of a graph $G$ is the degeneracy set of some stability condition \cite[Question 4.23]{fava-pagani-viviani2025}. The bijection between subdivisions of $\Lambda(\M^\ast(G))$ and extensions of $\M(G)$ yields the following corollary of Theorem \ref{thm:stability-condition-through-refinement}.
	\begin{corollary}\label{cor:stab->odeg is surj}
		Let $\text{VStab}(G)$ be the set of all $V-$stability conditions on $G$, and let $\mathcal{O}\text{Deg}(G)$ be the set of all orientable degeneracy sets of $G$. Then the map $\text{VStab}(G) \to \Or\text{Deg}(G)$ is surjective.
	\end{corollary}
	
	Even when restricting to classical stability conditions, we can have stability conditions that cannot come from a subdivision of $\Lambda(\M^\ast(G))$. 
	\begin{example}
		We consider the following stability condition associated to the numerical polarization on $K_4$ in Example 5 in Section 8 of \cite{oda-seshadri}, where the vertices of $K_4$ are labeled as in Example \ref{ex:not-all-extensions-orientable}. Let $\phi$ be the polarization on $K_4$ given by $\phi(v_1) = \phi(v_4) = 1$, and $\phi(v_2)= \phi(v_3) = 0$. The stability condition $\mathfrak{n}(\phi)$ has degeneracy set $\mathcal{D}(\mathfrak{n}(\phi)) = \{v_1v_2, v_1v_3, v_1v_4, v_2v_3, v_2v_4, v_3v_4\}$. $\mathcal{D}(\mathfrak{n}(\phi))$ corresponds to the linear subclass of cocircuits $\mathcal{L}$ discussed in Example \ref{ex:not-all-extensions-orientable}, and the corresponding single-element extension of $M(K_4)$ was not orientable.
	\end{example}
	
	Therefore, to resolve Question 4.23 in \cite{fava-pagani-viviani2025}, one must only consider the non-orientable degeneracy sets on $G$. Determining the orientability of a degeneracy set $\mathcal{D}$ is equivalent to determining if a matroid is orientable, which is NP-Complete \cite{orientable-complexity}. However, determining the orientability of a degeneracy set $\mathcal{D}$ can viewed as the existence of a relation in the poset $\text{Deg(G)}$.
	
	\begin{lem}\label{lem:odeg-upper-lifting}
		Consider degeneracy sets $\mathcal{D}_1$ and $\mathcal{D}_2$ such that $\mathcal{D}_1 \geq \mathcal{D}_2$. If $\mathcal{D}_2$ is orientable, then $\mathcal{D}_1$ is orientable. 
	\end{lem}
	\begin{proof}
		If $\mathcal{D}_2$ is orientable, there exists a subdivision $\Su_2$ of $\Lambda(\M^\ast(G))$ such that $\mathcal{D}_{\Su_2} = \mathcal{D}_2$ and a corresponding localization $\sigma_2^\ast$. If $\mathfrak{n}^2$ is the stability condition corresponding to $\Su_2$, recall from the proof of Theorem \ref{thm:subdiv-to-forest} that $\mathfrak{n}^2$ is defined by
		\[
		\mathfrak{n}^2_W = \begin{cases}
			g(G[W]) - 1 & W \in \mathcal{D}_2 \\
			g(G[W]) & \sigma_2^\ast(\vec{C}^\ast_W) = + \\
			g(G[W]) - 1 & \sigma_2^\ast(\vec{C}^\ast_W) = -.
		\end{cases}
		\]
		By Proposition \ref{prop:order-preserving-upper-lifting}, there exists a stability condition $\mathfrak{n}^1$ such that $\mathfrak{n}^1 \geq \mathfrak{n}^2$ and $\mathcal{D}(\mathfrak{n}^1) = \mathcal{D}_1$. We show that $\mathfrak{n}^1$ corresponds to some subdivision of the Lawrence polytope $\Lambda(\M^\ast(G))$. We recall from Equation \ref{eqn:upper-lifting-stab-condition} that $\mathfrak{n}_1$ is defined by 
		\[\mathfrak{n}^1_W = \begin{cases}
			\mathfrak{n}^2_W & W \in \mathcal{D}_1 \text{ or } W \notin \mathcal{D}_2 \\
			\mathfrak{n}^2_W + 1 & W \in \mathcal{E} \\
			\mathfrak{n}^2_W & W \in \mathcal{E}^c,
		\end{cases}\]
		where $\mathcal{E}$ is a subset of $\mathcal{D}_2 \backslash \mathcal{D}_1$ satisfying the conditions in Definition \ref{def:poset-relation-deg-set}. In particular, we can explicitly describe $\mathfrak{n}^1$ using the localization $\sigma_2^\ast$:
		\[\mathfrak{n}^1_W = \begin{cases}
			g(G[W]) - 1& W \in \mathcal{D}_1 \\
			g(G[W]) - 1 & W\notin \mathcal{D}_2 \text{ and } \sigma_2^\ast\vec{C}^\ast_W) = - \\
			g(G[W])  & W\notin \mathcal{D}_2 \text{ and } \sigma_2^\ast\vec{C}^\ast_W) = + \\
			g(G[W]) & W \in \mathcal{E} \\
			g(G[W]) - 1 & W \in \mathcal{E}^c.
		\end{cases}\]
		
		We define a map $\sigma^\ast_1:\mathcal{C}^\ast(\M(G)) \to \{+,-,0\}$ by 
		\[\sigma^\ast_1(\vec{C}^\ast_W) = \begin{cases}
			0 & W \in \mathcal{D}_1 \\
			+ &  W \notin \mathcal{D}_1 \text{ and } \mathfrak{n}^1_W = g(G[W]) \\
			- & W \notin \mathcal{D}_1 \text{ and }\mathfrak{n}^1_W = g(G[W]) - 1.
		\end{cases}\]
		Because $\mathfrak{n}^1$ is a stability condition of degree $g(G) - 1$, it follows that $\sigma^\ast_1(\vec{C}^\ast_W) = -\sigma^\ast_1(\vec{C}_{W^c}^\ast) = \sigma^\ast_1(-\vec{C}_W^\ast)$. We also note that $\sigma_2^\ast(\vec{C}^\ast_{W}) \in \{0, \sigma_1^\ast(\vec{C}^\ast_{W})\}$ for every $W \in \text{BCon}(G)$. 
		
		We then show that $\sigma_1^\ast$ is a localization by showing it does not produce one of the three forbidden subconfigurations in Figure \ref{fig:forbiddenconfigs}. Again, let $W_i$, $W_j$, and $W_k$ be pairwise disjoint biconnected subsets such that $W_i \cup W_j \cup W_k = V(G)$, and consider the rank two minor of $\M(G)$ with vertices $W_i$, $W_j$, and $W_k$, as pictured in Figure \ref{fig:matroidminor}. Since $\mathcal{D}_1$ is a degeneracy set, $\sigma^\ast_1$ cannot produce the third forbidden subconfiguration. We then consider the second forbidden subconfiguration. Assume without loss of generality that $\sigma_1^\ast(\pm \vec{C}_{W_i}^\ast) = 0$, so $W_i \in \mathcal{D}_1$. The relation $\mathcal{D}_1 \geq \mathcal{D}_2$ implies $\mathcal{D}_1 \subseteq \mathcal{D}_2$, so $W_i \in \mathcal{D}_2$ and $\sigma^\ast_2(\pm \vec{C}_{W_i}^\ast) = 0$. If $\sigma^\ast_1$ produces the second forbidden subconfiguration, we have $\sigma^\ast_1(\vec{C}_{W_j}^\ast) = \sigma^\ast_1(\vec{C}_{W_k}^\ast) = -$; however, $\sigma^\ast_2$ is a localization and cannot produce a forbidden subconfiguration, so we must have $\sigma^\ast_2(\pm \vec{C}_{W_j}^\ast) = \sigma^\ast_2(\pm \vec{C}_{W_k}^\ast) = 0$. Therefore $W_j$ and $W_k$ are both in $\mathcal{D}_2$ and both not in $\mathcal{D}_1$, so $W_j$ and $W_k$ are in $\mathcal{D}_2 \backslash \mathcal{D}_1$ and both are in $\mathcal{E}^c$. But, this contradicts the definition of $\mathcal{E}$, because exactly one of $W_j$ and $W_k$ must be in $\mathcal{E}$. Therefore, $\sigma^\ast_1$ cannot produce the second forbidden subconfiguration. Finally, if $\sigma^\ast_1$ produces the first forbidden subconfiguration, we must have $\sigma_2(\pm \vec{C}_{W_i}^\ast) = \sigma_2(\pm \vec{C}_{W_j}^\ast) = \sigma_2(\pm \vec{C}_{W_k}^\ast) = 0$ because $\sigma_2^\ast$ is a localization. However in this subconfiguration, we would then have $\sigma_1( \vec{C}_{W_i}^\ast) = \sigma_1( \vec{C}_{W_j}^\ast) = \sigma_1( \vec{C}_{W_k}^\ast)$, and either all three of $W_i, W_j, W_k$ appear in $\mathcal{E}$ or none of them do, contradicting the definition of $\mathcal{E}$.
		
		Therefore, $\sigma^\ast_1$ is a localization, and it corresponds to a subdivision $\Su_1$ of $\Lambda(\M^\ast(G))$ such that $\mathcal{D}_{\Su_1} = \mathcal{D}_1$. By definition, $\mathcal{D}_1$ is orientable. 
	\end{proof}
	
	\begin{thm}\label{thm:orientable-iff-bcon}
		The degeneracy set $\mathcal{D}$ is orientable if and only if $\mathcal{D}\geq \text{BCon}(G)$ in $\text{Deg}(G)$. 
	\end{thm}
	\begin{proof}
		First, assume that the degeneracy set $\mathcal{D}$ is orientable, so there exists a subdivision $\Su$ such that $\mathcal{D}_\Su = \mathcal{D}$. Then, consider the stability condition $\mathfrak{n}^\Su$, as defined in Theorem \ref{thm:stability-condition-through-refinement}. Recall the stability condition $\mathfrak{n}^{\deg}$ from Example \ref{ex:ndeg}, which satisfies $\mathfrak{n}^\Su \geq \mathfrak{n}^{\deg}$. Because the map from stability conditions to degeneracy sets is order-preserving, we have $\mathcal{D} = \mathcal{D}(\mathfrak{n}^\Su) \geq \mathcal{D}(\mathfrak{n}^{\deg}) = \text{BCon}(G)$, as desired.
		
		Then, assume $\mathcal{D}\geq \text{BCon}(G)$. By Lemma \ref{lem:odeg-upper-lifting}, the orientability of $\text{BCon}(G)$ implies that $\mathcal{D}$ is orientable. 
	\end{proof}
	
	\section{Tool and Computational Resource Disclosure}
	The author did not use any AI or LLM tools for any aspect of this research or the writing of this manuscript.

\end{document}